\documentclass{amsart}
\usepackage{xypic,amsthm,amssymb,hyperref,amsmath,graphicx,stmaryrd,boxedminipage,mathrsfs,manfnt, graphicx,mathtools,tikz,wrapfig,url,comment}
\usepackage[foot]{amsaddr}
\usetikzlibrary{cd}
\usepackage[all]{xy}
\usepackage[mathcal]{euscript}

\newcommand{\R}{\mathbf{R}}
\newcommand{\C}{\mathbf{C}}
\newcommand{\Z}{\mathbf{Z}}
\newcommand{\Q}{\mathbf{Q}}
\newcommand{\sma}{\wedge}

\newcommand{\ext}{\operatorname{Ext}}

\renewcommand{\hom}{\operatorname{Hom}}

\newcommand{\coker}{\operatorname{coker}}

\newcommand{\id}{\operatorname{Id}}

\newcommand{\Sq}{\operatorname{Sq}}
\newcommand{\sq}{\operatorname{Sq}}

\newcommand{\mc}{\mathcal}

\newcommand{\mbf}{\mathbf}

\newcommand{\Spin}{\operatorname{Spin}}
\newcommand{\Pin}{\operatorname{Pin}}

\newtheorem{thm}{Theorem}[section]
\newtheorem{lem}[thm]{Lemma}

\newtheorem{prop}[thm]{Proposition}

\theoremstyle{definition}
\newtheorem{defn}[thm]{Definition}
\newtheorem{example}[thm]{Example}
\newtheorem{rmk}[thm]{Remark}

\numberwithin{equation}{section}
\numberwithin{figure}{section}

\def\RPn (#1,#2){
  \fill (#1, #2) circle (3pt);
  \fill (#1, #2+1) circle (3pt);
}

\def\sqtwoL (#1,#2,#3){
  \draw[#3] (#1,#2) .. controls (#1-1,#2+1) .. (#1,#2+2);
}

\def\sqtwoR (#1,#2,#3){
  \draw[#3] (#1,#2) .. controls (#1+1,#2+1) .. (#1,#2+2);
}

\def \sqtwoCR (#1,#2,#3){
   \draw[#3] (#1,#2) .. controls (#1+1,#2+.5) and (#1+1.5,#2+2) .. (#1+2,#2+2);
}

\def \sqtwoCL (#1,#2,#3){
   \draw[#3] (#1,#2) .. controls (#1-1,#2+.5) and (#1-1.5,#2+2)  .. (#1-2,#2+2);
}

\def \sq1 (#1,#2,#3){
  \draw[#3] (#1,#2) -- (#1,#2+1);
}

\def\A1 (#1,#2){
\fill (#1, #2) circle (3pt);
\fill (#1, #2+1) circle (3pt);
\fill (#1, #2+2) circle (3pt);
\fill (#1, #2+3) circle (3pt);
\fill (#1+2, #2+3) circle (3pt);
\fill (#1+2, #2+4) circle (3pt);
\fill (#1+2, #2+5) circle (3pt);
\fill (#1+2, #2+6) circle (3pt);
\draw (#1, #2) -- (#1, #2+1);
\draw (#1, #2+2) -- (#1, #2+3);
\draw (#1+2, #2+3) -- (#1 + 2, #2+4);
\draw (#1+2, #2+5) -- (#1+2, #2+6);
\draw (#1, #2) .. controls (#1-1, #2+1) .. (#1, #2+2);
\draw (#1+2, #2+4) .. controls (#1+3, #2+5) .. (#1+2, #2+6);
\draw (#1, #2+1) .. controls (#1+1, #2+1.5) and  (#1+1.5 ,#2+3) .. (#1+2,#2+3);
\draw (#1, #2+2) .. controls (#1+1, #2+2.5) and (#1+1.5, #2+4) .. (#1+2, #2+4);
\draw (#1, #2+3) .. controls (#1+1, #2+3.5) and (#1+1.5, #2+5) .. (#1+2, #2+5);
}

\def\joker (#1,#2){
  \foreach \y in {#2, #2+1, #2+2, #2+3, #2+4}
           {\fill (#1,\y) circle (3pt);}
           \draw (#1,#2) -- (#1, #2+1);
           \draw (#1,#2+3) -- (#1, #2+4);
           \draw (#1,#2+0) .. controls (#1-1,#2+1) .. (#1, #2+2);
           \draw (#1,#2+2) .. controls (#1-1,#2+3) .. (#1, #2+4);
           \draw (#1,#2+1) .. controls (#1+1,#2+2) .. (#1, #2+3);
}

\title{Homotopy Theoretic Classification of Symmetry Protected Phases}
\author{Jonathan A. Campbell}
\address{Vanderbilt University \\ Department of Mathematics}
\email{j.campbell@vanderbilt.edu}
\begin{document}
\maketitle
\begin{abstract}We classify a number of symmetry protected phases using Freed-Hopkins' homotopy theoretic classification. Along the way we compute the low-dimensional homotopy groups of a number of novel cobordism spectra. \end{abstract}

\section{Introduction and Outline}

\subsection{Introduction}

Recently, symmetry protected topological phases (SPTs) have received a great deal of attention. Not only are they interesting phases of matter outside of the Landau symmetry breaking classification, but their realizations in nature would have applications to, for example, quantum computation. Very roughly, two systems are in the same SPT phase if their Hamiltonians are gapped, have an action by a group $G$, and can be smoothly deformed into one another equivariantly without closing the gap. Furthermore, one requires that if we remove the restriction of equivariance the systems can be deformed to a trivial state. There are various definitions of what SPTs and topological phases are in general (see, e.g. \cite{freedman_phases} for a discussion of definitions), but whatever the appropriate definition, it is widely supposed that the long-range behavior of an SPT is governed by an invertible topological quantum field theory (TQFT). 

There are many classification schemes for SPTs, but once the assumption that they are governed by TQFTs is accepted, homotopy theory may be brought to bear. Segal \cite{segal} and Atiyah \cite{atiyah} give a definition of TQFTs based on cobordism with elaborations due to Baez-Dolan \cite{baez_dolan} and Hopkins-Lurie \cite{hopkins_lurie} phrased in terms of $(\infty,n)$-categories --- these latter descriptions are usually called ``The Cobordism Hypothesis''. The $(\infty,n)$-category treatment requires formidable technical machinery, but it becomes much more manageable for \textit{invertible} field theories. In that case, an invertible field theory becomes a map of spectra (in the sense of algebraic topology) from $MTG(n)$, a spectrum built from $G$-cobordant $n$-manifolds, to $\Sigma^n I\C^\times$, the Brown-Comenetz dual of the sphere spectrum (for discussions of both of these objects, and further motivation, see \cite{freed_SRE, freed_hopkins}). Thus, a reasonable definition of a topological invertible field theory is simply a spectrum map $MTG(n) \to \Sigma^n I \C^\times$. This exists wholly within stable homotopy theory, and is amendable to attack by homotopy theoretic techniques. We have not yet mentioned any physical assumptions on the theory --- field theories are usually assumed to be unitary. A main goal of \cite{freed_hopkins} is implementing unitarity for invertible field theories and then classifying all deformation classes of such, which should correspond to classifying SPTs. They obtain the following classification scheme. 


\begin{thm}[Hopkins-Freed]
  Deformation classes of invertible, reflection positive, topological quantum field theories of space-time dimension $n$ with structure group $G$ are in bijection with homotopy classes of maps
  \[
  [MTG, \Sigma^{n+1} I_\Z]
  \]
\end{thm}

\begin{rmk}
Here $MTG$ is the cobordism spectrum of manifolds where the \textit{stable} tangent bundle has structure group $G$. In the case of many groups of interest, this is the same as $MG$. For example, $MTO = MO$, $MTSO = MSO$, $MT\Spin = M\Spin$. However, $MT \Pin^+ = MPin^-$ and $MTPin^- = M\Pin^+$. The spectrum $I_{\Z}$ is the Anderson dual of the sphere spectrum, this is explained more fully in Sect. \ref{notation}
\end{rmk}

Given these ingredients, the goal of this paper is to explicate the computations recorded in Freed-Hopkins \cite{freed_hopkins} and to extend the computations to other situations of physical interest. The results, along with the physically relevant names, are summarized below. We note that the computations below are in total agreement with computations done by physicists, which are done via very different methods. For this reason, we regard the computations below as a sort of ``experimental verification'' of the cobordism hypothesis --- the cobordism hypothesis is capturing something about physics. The reason for the agreement between various approaches is still very mysterious, however. In order to get at the problem, we would have to have a better understanding of how lattice models ``flow'' to topological field theories. The author hopes to return to this in future work. 

\begin{thm}[Hopkins-Freed]
  We have the following low-dimensional homotopy groups (for notation, see Sect. \ref{freed_hopkins_computations})
  \begin{center}
  \begin{tabular}{c|cccccc}
    $\ast$ & $\pi_0$ & $\pi_1$ & $\pi_2$ & $\pi_3$ & $\pi_4$ & $\pi_5$ \\\hline
    $MT\Pin^-$& $\Z/2$& $0$  &  $\Z/2$  & $\Z/16$  &  $0$  & $0$\\
    $MT\Pin^+$ &$\Z/2$ & $\Z/2$   & $\Z/8$   &  $0$  & $0$  & $0$\\
    $MT\Pin^{\tilde{c}-}$  &$\Z/2$   & $0$   & $\Z \times \Z/2$   & $0$  & $\Z/2$ & $0$\\
    $MT\Pin^{\tilde{c}+}$  &$\Z/2$   & $0$   & $\Z$  &$\Z/2$   & $(\Z/2)^3$ & $0$ \\
    $MTG^+$  & $0$  & $0$    & $\Z/2$  & $0$   & $\Z/2 \times \Z/4$ & $0$\\
  \end{tabular}
  \end{center}
\end{thm}

\begin{rmk}
There are other cobordism groups considered in \cite{freed_hopkins}, but they can be computed by exactly the same methods outlined below. 
\end{rmk}

\begin{rmk}
The computation of $M\Pin^-$ and $M\Pin^+$ is classical \cite{kirby_taylor}, but computed differently in the literature. Also, with the exception cases of $\Pin^{-}$ and $\Pin^+$, the author does not know of manifold representatives of the cobordism groups. For the purposes of physics (e.g. concrete lattice models) it would be useful and interesting to know such representatives. 
\end{rmk}

\begin{thm}---
  \begin{enumerate}
  \item Bosonic Systems without time reversal and with $U(1)$ symmetry. These are classified in $n = (d+1)$ spacetime dimensions by $[MSO \sma \C P^\infty_+, \Sigma^{n+1} I_{\Z}$. The classification in low dimensions is given by 
    \begin{center}
        \begin{tabular}{c|ccccc}
      $d $ & $0$ & $1$ & $2$ & $3$ & $4$ \\\hline
      $\ast$    & $\Z$     & $0$     & $\Z \oplus \Z$    & $0$     &  $ \Z\oplus\Z\oplus \Z/2$
        \end{tabular}
    \end{center}
  \item Bosonic systems with time reversal symmetry. This classification corresponds to a computation of $[MO,\Sigma^{n+1}I_\Z]$. In low degrees this is
    \begin{center}
      \begin{tabular}{c|ccccc}
      $d $ & $0$ & $1$ & $2$ & $3$ & $4$ \\\hline
      $\ast$    & $0$     & $\Z/2$     & $0$    & $(\Z/2)^2$     &  $ \Z/2$
      \end{tabular}
    \end{center}
  \item Bosonic systems with time reversal and $U(1)$ symmetry. This classification corresponds to $[MO \sma \C P^\infty_+, \Sigma^{n+1}I_\Z]$. In low dimensions we have
        \begin{center}
      \begin{tabular}{c|ccccc}
      $d $ & $0$ & $1$ & $2$ & $3$ & $4$ \\\hline
      $\ast$    & $0$     & $(\Z/2)^2$  & $0$   & $(\Z/2)^4$    & $\Z/2$
      \end{tabular}
    \end{center}
  \end{enumerate}
\end{thm}

We also have classification results for various fermionic systems

\begin{thm}---
  \begin{enumerate}
  \item Fermionic systems with an internal $\Z/2$ symmetry in dimension $(3+1)$. This classification corresponds to a computation of $[M\Spin \sma BZ/2_{+}, \Sigma^5 I_{\Z}]$, which vanishes. The same result holds for $\Z/2 \times \Z/2$ symmetry.
  \item There is a non-trivial fermionic phase in dimension $(3+1)$ with symmetry groups $\Z/2\times \Z/4$
  \item Fermionic systemss with an internal $(\Z/2)^{\times k}$ symmetry in dimension $(2+1)$. This corresponds to a compuation of $[M\Spin \sma (B\Z/2)^{\times k}_+, \Sigma^4 I_{\Z}]$ which is $(\Z/8)^k \oplus (\Z/4)^{\binom{k}{2}} \oplus (\Z/2)^{\binom{k}{3}}$
  \end{enumerate}
\end{thm}

Finally, we include the low-dimensional homotopy groups of a family of novel cobordism spectra. These should be thought of as $\Spin^c$, but with the $U(1)$ broken down to $\Z/2^n$. 

  \begin{thm}
 A fermionic system with $\Z/2^n$ symmetry where the central $\Z/2$ squares to fermionic parity. This amounts to a computation of $\pi_\ast (M\Spin \times_{\Z/2} \Z/2^n)$. In low dimensions this is    \begin{align*}
      \pi_0 M(\Spin \times_{\Z/2} \Z/2^n) &= \Z \\
      \pi_1 M(\Spin \times_{\Z/2} \Z/2^n) &= \Z/2^n\\
      \pi_2 M(\Spin \times_{\Z/2} \Z/2^n) &= 0\\
      \pi_3 M(\Spin \times_{\Z/2} \Z/2^n) &= \Z/2^{n-2}\\
      \pi_4 M(\Spin \times_{\Z/2} \Z/2^n) &= \Z      \\
    \end{align*} 
  \end{thm}

  \begin{rmk}
    Again, it would be interesting to know manifold representatives of these classes. 
  \end{rmk}

The computations involve ingredients from homotopy theory such as the Steenrod algebra, the Adams spectral sequence, the Atiyah-Hirzebruch spectral sequence, etc. In order to have this paper be somewhat accessible to physicists, we have included  background material on these concepts.

\subsection{Outline}

In Section 2 we introduce the Steenrod algebra. This is a graded algebra which acts on every cohomology group $H^\ast (X)$ for any space $X$. This action endows $H^\ast (X)$ with a rich module structure that allows for numerous useful computations. It is also the foundation for the Adams spectral sequence, the primary computational tool in this paper. We give examples of the utility of the Steenrod algebra, as well as examples of the Steenrod module structure of various spaces that will be used throughout the paper.

In Section 3 we specialize to considering a particular sub-algebra of the Steenrod algebra, dubbed $\mc{A}(1)$. The reasons for considering this algebra are manifold. Primarily for us, the algebra is closely related to the homological properties of $KO$-theory. Many of the computations in this paper will essentially be computations of low-dimensional Spin-cobordism, i.e. $MSpin^\ast (X)$ for $\ast \leq 5$. In this range $MSpin$ and $KO$ are isomorphism with the isomorphism given by the Atiyah-Bott-Shapiro map $MSpin \to ko$. Thus, low-dimensional Spin cobordism computations boil down to computations of $ko^\ast (X)$. As we will explain later, this computation reduces to understanding the $\mc{A}(1)$-module action on $H^\ast (X)$. It is thus necessary to understand $\mc{A}(1)$-modules.

Section 4 is a review of some necessary homological algebra. All of this is standard, but it is useful to have necessary results collected in one place.

Section 5 is devoted to the Adams spectral sequence, a spectral sequence whose input is the $\mc{A}$-module structure of $H^\ast (X)$ and whose output is (approximately) the homotopy group $\pi_\ast (X)$ We review the statement of the Adams spectral sequence, as well as some example computations. Almost all of our examples will be computating related to computing $ko$-homology , which amounts to computing the homotopy groups $\pi_\ast (ko \sma X)$.  As we will see, this will involve homological algebra over $\mc{A}(1)$.

In section 6 we carefully illustrated how the computations in \cite{freed_hopkins} are carried out. Much of the work is computing $\mc{A}(1)$-module actions on various Thom spaces.

Section 7 is the real meat of the paper and is devoted the the classification of various SPT phases and cobordism groups.

\subsection{Prerequisites, Conventions, Notation}\label{notation}

This paper makes heavy use of stable homotopy theory. We try to provide as much background for computations as possible, but there is not room to review spectra. For an excellent introduction, see \cite{adams_stable_homotopy}. In most cases, the reader will not be hurt reading ``spaces'' for ``spectra''.

Unless indicated otherwise, in the entire paper we will in cohomology with $\Z/2$-coefficients. In most of the paper we also work \textit{2-locally} or \textit{2-adically}. By deep work of Sullivan \cite{sullivan_localization} and Bousfield \cite{bousfield_localization_I,bousfield_localization_II} for any space or spectrum $X$ and any prime $p$ there exists a topological space $X_{(p)}$ called the $p$-localization such that there is a map $X \to X_{(p)}$ which induces an equivalence $H^\ast (X_{(p)};\Z_{(p)}) \xrightarrow{\cong} H^\ast (X; \Z_{(p)})$ and $X_{(p)}$ is the universal space with this property. That is, given any $X \to  Y$ which is an equivalence on $\Z_{(p)}$ cohomology, it factors through $X \to  X_{(p)}$. Similarly, one can define a $p$-adic completion of a space $X$, denoted $X^{\sma}_p$ by replacing $\Z_{(p)}$ by $\Z/p$ throughout the above definition. Below, we work with the Adams spectral sequence over the mod $2$ Steenrod algebra, which converges to a 2-adic completion of the spaces involved. This presents us with no difficulties because the spaces involved lack torsion away from 2. However, in the sequel we often gloss over this point in order to not complicate the exposition.

The spectrum $I_\Z$, the Anderson dual of the sphere spectrum is used throughout the paper, and  can be defined as follows (see \cite[App. B]{hopkins_singer}). We consider cohomology theories given on a spectrum $X$ by $\hom(\pi_\ast X, \Q)$ and $\hom(\pi_\ast X, \Q/\Z)$ (one can check that these satisfy the Eilenberg-Steenrod axioms and so, indeed, define cohomology theories). The first is just rational cohomology and so is represented by the Eilenberg-MacLane spectrum $H\Q=:I_{\Q}$. The second is represented by the Brown-Comenetz dual of the sphere \cite{brown_comenetz}, denoted $I_{\Q/\Z}$. The map $\Q \to \Q/\Z$ induces a map on representing spectra $I_{\Q} \to I_{\Q/\Z}$. The Anderson dual, $I_{\Z}$ is the fiber of this map. 

The notation $X\langle a, \dots, b\rangle$ will denote a topological space truncated so that $X$ only has non-zero Postnikov sections between degrees $a$ and $b$.

\subsection{Acknowledgements} 

I would like to heartily thank Dan Freed for suggesting this project and for insight at many points during it. Conversations with John Morgan and Greg Brumfiel about $M(\Spin\times_{\Z/2} \Z/2^n)$ provided the impetus for the inclusion of that computation. Finally, Agn\`{e}s Beaudry read a draft of this paper and caught a number of inaccuracies and typos; I'm indebted to her for her careful reading.

\section{The Steenrod Algebra}

The Steenrod algebra is a fundamental tool in homotopy theory and the cornerstone of the Adams spectral sequence. In this section we give a brief introduction to the Steenrod algebra. There are many excellent sources for this material in the literature \cite{mosher_tangora,steenrod_epstein,hatcher,rognes}. A more physically motivated introduction is given in \cite[Sect. 4]{witten_e8}. 

Succinctly, the Steenrod algebra is the algebra of stable cohomology operations on $\Z/2$-cohomology. Here a cohomology operation is a natural transformation of functors $H^a (-; \Z/2) \to H^b (-; \Z/2)$ (for some $a, b$). The statement that the operations are stable means that they must commute with suspension, i.e. the following diagram must commute:
\begin{equation}
\xymatrix{
  H^a (-;\Z/2) \ar[r]\ar[d]_{\Sigma} & H^b (-;\Z/2)\ar[d]^{\Sigma}\\
  H^{a+1} (\Sigma - ; \Z/2) \ar[r] & H^{b+1} (\Sigma -;\Z/2). 
}
\end{equation}

Steenrod constructed a family of operations in mod 2 cohomology, one for each natural number $i$, $H^{\ast} (X; \Z/2) \to H^{\ast + i}(X; \Z/2)$. Their properties are as follows (the construction of these operations is fairly technical)

\begin{thm}\cite{steenrod_epstein}
  There are cohomology operations $\Sq^i: H^\ast (X; \Z/2) \to H^{\ast + i}(X; \Z/2)$ with the follow properties:
  \begin{enumerate}
  \item $\Sq^i$ is stable
  \item $\Sq^i$ is natural. That is, given $f: X \to Y$ we have
    \[
    f^\ast \Sq^i = \Sq^i f^\ast 
    \]
  \item $\Sq^0 = \id$
  \item $\Sq^n (x) = x\smile x$ for $x \in H^n (X; \Z/2)$
  \item $\Sq^i$ vanishes on cohomology classes of degree $< i$
  \item The Cartan formula holds:
    \begin{equation}\label{cartan_formula}
    \Sq^i (x \smile y) = \sum_{m + n = i} \Sq^m x \smile \Sq^n y 
    \end{equation}
  \end{enumerate}
\end{thm}

The above can be taken as a characterization of the Steenrod operations. 

The Steenrod operations generate an algebra, appropriately called the Steenrod algebra, and denoted $\mc{A}$,  where multiplication is composition of operations, so that the algebra is composed of operations $\Sq^{i_1} \cdots \Sq^{i_n}$. There are key relations in this algebra. 

\begin{prop}[The Adem Relations]
  The following hold in $\mc{A}$: for $a < 2b$
  \begin{equation}\label{adem_relations}
  \Sq^a \Sq^b = \sum^{\lfloor a /2\rfloor}_{c = 0} \binom{b - c - 1}{a - 2c} \Sq^{a+b-c} \Sq^c 
  \end{equation}
\end{prop}

\begin{rmk}
This means that if one has $\Sq^a \Sq^b$ where $a < 2b$, for example $\Sq^3 \Sq^2$, $\Sq^7 \Sq^4$, $\Sq^1 \Sq^2$, etc, then the expression can be decomposed.
\end{rmk}

\begin{example}[Adem Relations]
  We have
  \begin{align}
    \Sq^3 \Sq^2 &= 0 \\
    \Sq^1 \Sq^2 &= \Sq^3 \\
    \Sq^2 \Sq^5 &= \Sq^6 \Sq^1 \\
    \Sq^2 \Sq^3 &= \Sq^4 \Sq^1 + \Sq^5
  \end{align}
\end{example}

It is a theorem of Serre that the algebra of stable cohomology operations is generated by the Steenrod operations subject to the Adem relations. Note that this algebra is decidedly non-commutative.

\begin{thm}[Serre]
We call a sequence $I = (i_1, \dots, i_n)$ \textbf{admissible} if $i_{j} \geq 2 i_{j+1}$. Let $\Sq^I$ denote $\Sq^{i_1} \cdots \Sq^{i_n}$. Then stable cohomology operations are generated (as a vector space) by $\Sq^I$ where $I$ is an admissible sequence. 
\end{thm}

\begin{rmk}
The basis $\Sq^I$ with $I$ admissible is called the \textbf{Serre-Cartan basis}.  
\end{rmk}

\begin{rmk}
Serre's theorem says that the Steenrod algebra encompasses \textit{all} stable mod 2 cohomology operations. 
\end{rmk}

The cohomology ring of any space has the structure of a module under the Steenrod algebra action. That is, no matter what the space, the cohomology ring is the module of an incredibly rich ring. This is very powerful.

\begin{example}
  The space with the easiest Steenrod algebra structure to understand is $\R P^n$. We of course know that $H^\ast (\R P^n) \cong \Z/2[x]/(x^{n+1})$, with $\Z/2$ coefficients in cohomology understood. We first note that the \textbf{total square} $\mbf{Sq} = 1 + Sq^1 + \Sq^2 + \cdots$ is a ring homomorphism. Thus
  \begin{equation}
  \mbf{Sq} (x^i) = (\mbf{Sq} x)^i = (x + \Sq^1 x + \Sq^2 x + \cdots)^i = (x + x^2)^i 
  \end{equation}
  where the last equality follows from the fact that $x$ is degree 1 so that $\Sq^1 x = x^2$ and $\Sq^j x$ for any $j > 1$ must vanish.
  
  Comparing coefficients we easily obtain
  \begin{equation}
  \Sq^j x^i = \binom{i}{j} x^{j+i}
  \end{equation}
  where we recall the binomal coefficients must be taken mod 2. 

\end{example}

\begin{example}
  A classical use of the cup product structure on cohomology is to distinguish the structure of $\C P^2$ and $S^2 \vee S^4$. In the former case the cohomology ring has a non-trivial multiplication, and in the latter it does not. That is, as rings we have $H^\ast (\C P^2) = \mbf{F}_2[x]/(x^3)$ where $\deg x = 2$ and $H^\ast(S^2 \vee S^4) \cong \Lambda(x) \oplus \Lambda(y)$ where $\deg x = 2$ and $\deg y = 4$.  The same technique cannot be used to distinguish $\Sigma \C P^2$ and $S^3 \vee S^5$ since they have the same cohomology ring, $\Lambda(x) \oplus \Lambda(y)$ with $\deg (x) = 3$, $\deg(y) = 5$.  However, the Steenrod algebra structure can. In $H^\ast (\C P^2)$, $\Sq^2 (x) = x^2$ because $x$ is of degree 2, and since the Steenrod operations are stable, $\Sq^2 (x) = y$ in $H^\ast (\Sigma \C P^2)$. However, in $H^\ast (S^2 \vee S^4)$, $\Sq^2 (x) = 0$, so the same relation holds in $H^\ast (S^3 \vee S^5)$. 
\end{example}

\begin{example}
  Let $f: S^{2n-1} \to S^n$ be a map of spheres with the indicated dimension. Using this map, one may glue a $2n$-cell to an $n$-cell to get a complex $e^n \cup_{f} e^{2n}$. In the cohomology ring of this complex, $H^\ast (e^n \cup_f e^{2n})$ there is obviously an element $x$ in degree $n$ and an element $y$ in degree $2n$, and there is some constant $h(f)$ such that $x^2 = h(f) y$. This $h(f)$ is called the \textbf{Hopf invariant} and is 0 or 1 (because we are working mod 2). The Hopf maps $f: S^3 \to S^2$, $S^7 \to S^4$ and $S^{15} \to S^8$ can be shown to be homotopically non-trivial using the Hopf invariant and the complex, quaternionic and octonionic projective spaces, respectively. It is a deep theorem due to Adams \cite{adams_hopf_inv_one} that these are the only cases where the Hopf invariant can be 1. However, using the Steenrod operations, one can see that $n$ must be a power of two. To see this, note that in $H^\ast (e^n \cup_f e^{2n})$, $x^2 = \Sq^n (x)$. However, by the Adem relations (or Serre's theorem), any $\Sq^n$ with $n \neq 2^i$ is decomposable. Since such a $\Sq^n$ is decomposable it must be that $\Sq^n x = 0$. 

\end{example}

\begin{example}[Bockstein Operations]
  The Bockstein operations are defined by short exact sequences of coefficient groups. For example, consider the short exact sequence
  \begin{equation}
  0 \to \Z/2 \xrightarrow{\times 2} \Z/4 \xrightarrow{\text{mod}\ 2} \Z/2 \to 0. 
  \end{equation}
  This induces a short exact sequence of cochain complexes, and thus a long exact sequence in cohomology groups:
  \begin{equation}
  \to H^i (X; \Z/2) \to H^i(X;\Z/4) \to H^i(X;\Z/2) \to H^{i+1}(X;\Z/2) \to \cdots. 
  \end{equation}
  The last arrow defines a cohomology operation $H^i(-;\Z/2) \to H^{i+1}(-;\Z/2)$ called the Bockstein operator, $\beta$. For this coefficient sequence, $\beta$ and the Steenrod operations $\Sq^1$ concide: $\beta = \Sq^1$.

  There are other coefficient sequences that produce operations that will be useful later. Consider
  \begin{equation}
  \Z \xrightarrow{\times m} \Z \xrightarrow{\text{mod}\ m} \Z/m
  \end{equation}
  which gives rise to a long exact sequence and thus a cohomology operation
  \[
  H^i(X;\Z/m) \to H^{i+1} (X;\Z).
  \]
  This is called the \textbf{integral Bockstein} and denoted $\widetilde{\beta}$. We will most often use this for the coefficient group $\Z/2$. 
\end{example}

In general, it is non-trivial to compute the full action of the Steenrod algebra on a cohomology ring. In this paper we'll be primarily concerned with Thom spaces and spectra. The following examples work toward computing the action of the Steenrod algebra in these cases. 

\begin{example}[Steifel-Whitney classes]
  Recall that the Stiefel-Whitney classes are invariants of vector bundles, measuring their non-triviality \cite{milnor_stasheff}. One method for the construction of the Stiefel-Whitney classes is via the Steenrod algebra. Let $E \to B$ be a vector bundle of rank $n$ and $U \in H^n (Th(E))$ the corresponding Thom class. Let $\mbf{Th}: H^\ast (B) \to H^{\ast + n} (Th(E))$ be the Thom isomorphism. Then the $i$th Stiefel-Whitney class $w_i$ is defined to be the class
  \begin{equation}
  w_i = \mbf{Th}^{-1} \Sq^i U \in H^i (B). 
  \end{equation}
  Note that the Thom isomoprhim $\mbf{Th}$ is given by taking the cup product with $U$, so the above also tells us the Steenrod action on $U$:
  \begin{equation}
  \Sq^i U = w_i \smile U
  \end{equation} 
This will be used heavily in the sequel. 
\end{example}

\begin{example}
  The above example explicitly computes the Steenrod operations on the Thom class of a vector bundle. One could also ask about the Stiefel-Whitney classes $w_i \in H^i (BO(n))$ themselves, where recall \cite[Sect.7]{milnor_stasheff} that $H^\ast (BO(n); \Z/2) \cong \Z/2[w_1, \dots, w_n]$ as rings. There is a convenient formula for the Steenrod action on the Stiefel-Whitney classes due to Wu:
  \begin{equation}
  \Sq^i (w_j) = \sum_{k = 0}^i \binom{(j-i) + (k-1)}{k} w_{i-k} w_{j+k}
  \end{equation}
  To prove this formula we use the splitting principle: the fact that any vector bundle can be pulled back to a space over which the bundle splits into line bundles. This turns out to given an isomorphism
  \begin{equation}
  s: H^\ast (BO(n); \Z/2) \to H^\ast ((\R P)^\infty; \Z/2)^{\Sigma_n} \cong \Z/2[x_1, \dots, x_n]^{\Sigma_n}
  \end{equation}
  where $s (w_j) = \sigma_j$ where $\sigma_j$ is the $i$th elementary symmetric polynomial in the $x_i$ (in fact, this is how the cohomology ring $H^\ast (BO(n); \Z/2)$ is computed). Naturality implies that its enough to compute $\Sq^i \sigma_j$ to determine $\Sq^i w_j$. 
\end{example}

\begin{example}
  The Steenrod operations on Thom spaces. Given $BO(n)$, there is a canonical bundle, $\gamma$, over it. We can take the Thom space of this bundle to obtain a space $MO(n)$. Let $U$ be the Thom class in $H^n (MO(n))$. We have, by the Thom isomorphism, an equivalence of vector spaces $H^\ast (BO(n)) \xrightarrow{\smile} H^\ast (MO(n))$. Thus, the elements of $H^\ast (MO(n))$ will look like polynomials on the $w_i$ cupped with $U$; for example, $w^2_1 U$, $(w_3 + w_1 w_2) U$, etc. We furthermore know the action of the Steenrod algebra on the Stiefel-Whitney classes, on the Thom class, and we have the Cartan formula. This allows us to completely compute the Steenrod action on $H^\ast (MO(n))$. As an example computation
  \begin{align}
    \Sq^3 (w_1 U) &= \Sq^3 (w_1) U + \Sq^2(w_1) \Sq^1 U + \Sq^1 w_1 \Sq^2 U + w_1 \Sq^3 U\\
    &= w^2_1 w_2 U + w_1 w_3 U = (w^2_1 w_2 + w_1 w_3) U
  \end{align}
  Again, computations like this will be used frequently in later sections of the paper. 
\end{example}

  \begin{example}
    The Steenrod operations on $MTO(n)$. This will be similar to the case of $MO(n)$, but the Thom class behaves differently. Recall the definition of $MTO(n)$ \cite{galatius_tillman_madsen_weiss}. We consider the Grassman manifold $\mbf{Gr}(n, m)$ of $n$-planes embedded in $\R^m$. There are two bundles over $\mbf{Gr}(n, m)$. There is the bundle $\gamma_{n,m}$ of $n$-planes together with a vector in that $n$ plane, and then there is the orthogonal complement $\gamma^{\perp}_{n, m}$, i.e. the bundle such that $\gamma_{n,m} \oplus \gamma_{n,m}^{\perp} \cong \mbf{Gr}_{n,m} \times \R^{n+m}$. When $m = \infty$, $\mbf{Gr}(n, m) \cong BO(n)$ and $\gamma_{n,m}$ is the bundle in the previous example. We define a spectrum whose $(m+n)$th space is $\mbf{Th}(\gamma^{\perp}_{m,n})$ --- we call this spectrum $MTO(n)$.

    To see the action of the Steenrod algebra, we note that since the bundles defining $MTO(n)$ are complementary to the bundles defining $MO(n)$, the total Stiefel-Whitney classes are complementary to each other (see, e.g. \cite{milnor_stasheff}). That is for $\gamma_{n,\infty}^{\perp} \to BO(n)$ we have
    \begin{equation}
    \mbf{w}(\gamma^\perp_{n,\infty}) = \frac{1}{\mbf{w}(\gamma_{n,\infty})} = \frac{1}{1+w_1+w_2+w_3+\cdots}
    \end{equation}
    Let $U^{\perp}$ denote the Thom class for this bundle. As in the computation above we find
    \begin{equation}
    \Sq^i (U^\perp) = \left(\frac{1}{1+w_1+w_2+\cdots}\right)_i \smile U^{\perp}
    \end{equation}
    where $(-)_i$ denotes the degree $i$ component of the formal sum.

    For example, for $BO(2)$ the bundle complementary to the canonical $2$-plane bundle has total Stiefel-Whitney class
    \begin{equation}
    \frac{1}{1+w_1+w_2} = 1 + w_1 + w_2 + w^2_1 + w^2_2 + w^3_1 + w^2_1 w_2 + w^2_2 w_1 + w^3_2 + \cdots 
    \end{equation}
    So we have
    \begin{equation}
    \Sq^1(U^\perp) = w_1 U^{\perp} \qquad \Sq^2 (U^\perp) = (w_2 + w^2_1)U^{\perp}
    \end{equation}
    The rest of the action can be computed as in the previous example using the Wu formula and Cartan formula. 
  \end{example}

\section{$ko$ and $\mc{A}(1)$}

The Steenrod algebra is clearly quite unruly: it is non-commutative and relations are complicated. It is sometimes convenient to work with smaller sub-algebras of the Steenrod algebra, and their use is in fact necessitated by natural considerations (see Section 5). 

For us, the main example of a such a sub-algebra will be $\mc{A}(1)$ which is the sub-algebra generated by $\langle \Sq^1, \Sq^2 \rangle$. This belongs to a family of sub-algebras $\mc{A}(n)$ which are generated by $\langle \Sq^1, \Sq^2, \cdots, \Sq^{2^n} \rangle$. However, we will not need these for $n > 1$ (for a picture of how complicated $\mc{A}(2)$ is, \cite{douglas_henriques_hill}). Of course, by restriction, $H^\ast(X)$ is also an $\mc{A}(1)$-module, and this shred of the $\mc{A}$ action still carries valuable information: 

\begin{thm}\label{cohomology_ko}
  We have $H^\ast (ko) \cong \mc{A} \sslash \mc{A}(1)$ (see \ref{double_slash} for notation) and 
  \begin{equation}
  H^\ast (ko \sma X) \cong \mc{A} \otimes_{\mc{A}(1)} H^\ast (X)
  \end{equation}
  as $\mc{A}$-modules. 
\end{thm}

That is, in some sense, $ko$ only ``sees'' the $\mc{A}(1)$ algebra structure of a space, $X$. In the sequel, when we are using the Adams spectral sequence to compute $ko$-theory spectra, this has the practical consequence that we will only need to know the $\mc{A}(1)$-action rather than the whole $\mc{A}$-action.

Motivated by the above, this section will be a quick review of the structure of $\mc{A}(1)$ and of some $\mc{A}(1)$-modules that will be useful for later computations. 

The following is a pleasant exercise in low-dimensional relations in the Steenrod algebra
\begin{prop}
  $\mc{A}(1)$ has $\Z/2$-basis
  \begin{equation}
  \Sq^0, \Sq^1, \Sq^2, \Sq^3, \Sq^2\Sq^1, \Sq^3\Sq^1, \Sq^5+\Sq^4\Sq^1, \Sq^5\Sq^1
  \end{equation}
\end{prop}
\begin{proof}
  We begin generating elements. $\Sq^1 \Sq^2 = \Sq^3$ by the Adem relations and $\Sq^2 \Sq^1$ is irreducible. Now, $\Sq^1 \Sq^3 = 0$, $\Sq^3 \Sq^1$ is irreducible, $\Sq^3 \Sq^2 = 0$ and $\Sq^2 \Sq^2 = \Sq^3 \Sq^1$.

  So far we have generated $\Sq^0, \Sq^1, \Sq^2, \Sq^3, \Sq^3 \Sq^1$.

  The remaining combinations to try are $\Sq^2 \Sq^3$ which is $\Sq^5 + \Sq^4 \Sq^1$, and $\Sq^2 (\Sq^3 \Sq^1) = \Sq^5 \Sq^1$. Finally, $\Sq^1 (\Sq^5 + \Sq^4 \Sq^1) = \Sq^5 \Sq^1$ and $\Sq^2 ( \Sq^5 + \Sq^4 \Sq^1) = 0$ and we're done. 
\end{proof}

The algebra $\mc{A}(1)$ is usually represented by a diagram as in Fig. \ref{a1}, where the straight lines represent left multiplication by $\Sq^1$ and any curved line represents left multiplication by $\Sq^2$. 
\begin{figure}
\begin{center}
\begin{tikzpicture}
  \draw (0,0) node[anchor=north](sq0){$\Sq^0$};
  \draw (0,1) node(sq1){$\Sq^1$};
  \draw (0,2) node(sq2){$\Sq^2$};
  \draw (0,3) node(sq3){$\Sq^3$};
  \draw (2,3) node(sq2sq1){$\Sq^2 \Sq^1$};
  \draw (2,4) node(sq3sq1){$\Sq^3 \Sq^1$};
  \draw (2,5) node(sq5plussq4sq1){$\Sq^5 + \Sq^4 \Sq^1$};
  \draw (2,6) node (sq5sq1){$\Sq^5\Sq^1$};
  \draw (sq0) -- (sq1);
  \draw (sq2) -- (sq3);
  \draw (sq2sq1) -- (sq3sq1);
  \draw (sq5plussq4sq1) -- (sq5sq1);
  \draw (sq0) .. controls (-1,1) .. (sq2);
  \draw (sq3sq1) .. controls (3.5, 5) .. (sq5sq1);
  \draw (sq1) .. controls (0,2) and (1.5, 2) .. (sq2sq1); 
  \draw (sq2) .. controls (0,3) and (1.5, 3) .. (sq3sq1);
  \draw (sq3) .. controls (0, 4) and (1.5, 4).. (sq5plussq4sq1);
  \A1 (4,0)
\end{tikzpicture}
\end{center}
\caption{\label{a1}A diagram of $\mc{A}(1)$} 
\end{figure} 

The rest of the section is devoted to examples of $\mc{A}(1)$-modules that will be used later. It also illustrates standard diagrammatic methods for $\mc{A}(1)$-modules. 

\begin{example}
  The $\mc{A}(1)$-module of $\R P^\infty$ is instructive. From earlier computations, we fully know the structure of the Steenrod module structure of $H^\ast (\R P^\infty)$: $\Sq^i (x^n) = \binom{n}{i} x^{n+i}$. So, we have
  \[
  \Sq^1 (x^n) = n x^{n+1} \qquad \Sq^2 (x^n) = \frac{n(n-1)}{2} x^{n+2}
  \]
  Note that when $n \equiv 0\ mod\ 4$ any element of $\mc{A}(1)$ will annhilate $x^n$. Given the two relations above we can easily draw a pictorial representation of the $\mc{A}(1)$-module $H^\ast (\R P^\infty)$ (see Fig. \ref{fig:rpinfinity}). 
  \begin{figure}
  \begin{center}
  \begin{tikzpicture}[scale=.5]
    \foreach \y in {0, 1, 2, 3, 4, 5, 6, 7, 8}
             {\fill (0, \y) circle (3pt);}
             \draw (0, 1) -- (0, 2);
             \draw (0, 2) .. controls (-1, 3) .. (0, 4);
             \draw (0, 3) -- (0, 4);
             \draw (0, 3) .. controls (1, 4) .. (0, 5);
             \draw (0, 5) -- (0, 6);
             \draw (0, 6) .. controls (1, 7) .. (0, 8);
             \draw (0, 7) -- (0, 8); 
  \end{tikzpicture}
  \end{center}
  \caption{\label{fig:rpinfinity}The $\mc{A}(1)$-module structure of $H^\ast (\R P^\infty)$}
  \end{figure}
  The figure should be interpreted as continuing upwards indefinitely. Every dot represents an element of the module. Straight lines represent an action by $\Sq^1$ and curved lines represent an action by $\Sq^2$. This is the standard pictorial way of presenting $\mc{A}(1)$-modules and will be used throughout. 
\end{example}

\begin{example}
We could also consider $H^\ast (\R P^n)$. This is the $\mc{A}(1)$-module $H^\ast (\R P^\infty)$ but truncated above degree $n$. Similarly, we could look at $H^\ast (\R P^n_k)$, where $\R P^n_k$ is the cell complex of $\R P^n$, but truncated below degree $k$. The cohomology of the truncated projective space, as an $\mc{A}(1)$-module, is obtained by truncating the $\mc{A}(1)$-module $H^\ast (\R P^n)$ below degree $k$. These modules are typically called $P^n_k$. It is easy to see from pictures that $\Sigma^4 P^n_k \simeq P^{n+4}_{k+4}$
\end{example}

Certain $\mc{A}(1)$-modules arise frequently enough to have been given standard names.

\begin{example}[Question marks] The ``upside-down question mark'', $Q$, is the $\mc{A}(1)$-module specified by the diagram in Fig. \ref{fig:question_mark}. 
  \begin{figure}
  \begin{center}
    \begin{tikzpicture}[scale=.5]
    \foreach \y in {0,2,3}
             {\fill (0,\y) circle (3pt);}
             \draw (0,2) -- (0,3);
             \draw (0,0) .. controls (-1,1) .. (0,2);
             \end{tikzpicture}
  \end{center}
  \caption{\label{fig:question_mark}The $\mc{A}(1)$-module $Q$}
  \end{figure}
\end{example}

\begin{example}[Joker]
  The joker is the $\mc{A}(1)$-module pictured in Fig. \ref{fig:joker} (turn the figure sideways to see a joker's cap), or any suspension thereof. It can be recognized as the $\mc{A}(1)$-module generated by $\Sq^2$.
  \begin{figure}
\begin{center}
    \begin{tikzpicture}[scale=.5]
      \foreach \y in {0, 1, 2, 3, 4}
               {\fill (0,\y) circle (3pt);}
               \draw (0,0) -- (0, 1);
               \draw (0,3) -- (0, 4);
               \draw (0,0) .. controls (-1,1) .. (0, 2);
               \draw (0,2) .. controls (-1,3) .. (0, 4);
               \draw (0,1) .. controls (1,2) .. (0, 3);
    \end{tikzpicture}
\end{center}
\caption{\label{fig:joker}The Joker}
\end{figure}
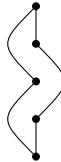
\end{example}

\begin{example}[Convenient resolution] \label{convenient_resolution}
  There is an $\mc{A}(1)$-module that arises frequently which is convenient for its homological properties. We will just present it here as another example of an $\mc{A}(1)$-module, pictured in Fig. \ref{fig:convenient_resolution}(the picture continues upward indefinitely)
  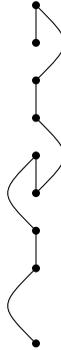
\begin{figure}
  \begin{center}
     \begin{tikzpicture}[scale = .50]
      \foreach \y in {0, 2, 3, 4, 5, 6, 7, 8,9}
               {\fill (0,\y) circle (3pt);}
               \draw (0,0) .. controls (-1, 1) .. (0, 2);
               \draw (0, 2) -- (0, 3);
               \draw (0, 3) .. controls (-1,4) .. (0, 5);
               \draw (0, 4) -- (0, 5);
               \draw (0, 4) .. controls (1, 5) .. (0, 6);
               \draw (0, 6) -- (0, 7);
               \sqtwoR(0,7,black);
               \draw (0,8) -- (0,9);
     \end{tikzpicture}
  \end{center}
  \caption{\label{fig:convenient_resolution}And $\mc{A}(1)$-module with a particularly nice resolution}
  \end{figure}
Note this module differs from $H^\ast (\R P^\infty)$ only in the bottom cells. 
\end{example}

\section{Some Graded Algebra}

In preparation for the discussion of the Adams spectral sequence, we need a discussion of  graded algebra and the concommitant notions of homological algebra. We will be brief since there is little difference between homological algebra and graded homological algebra; most of the difference lies in the fact that resolutions will be bi-graded in the latter case. Good references for this material are \cite{mccleary, ravenel}.

\begin{defn}
A \textbf{graded algebra} $A$ is an algebra where each element has a degree or grading $\deg (x) \in \mbf{Z}$. The grading effects the multiplication in that for $\deg(x) = p$ and $\deg(y) = q$, $x \cdot y = (-1)^{pq} y x$, the usual Koszul sign rule. 
\end{defn}
\begin{example}
The cohomology ring $H^\ast (X)$ of a topological space $X$ is a graded ring with grading giving by cohomological degree. 
\end{example}

\begin{rmk}
For the most part we will work over the field $\Z/2$ and sign rules will not rear their head. 
\end{rmk}

\begin{defn}
A \textbf{module} over a graded algebra $A$ is a graded vector space $M$ together with a map $A \otimes M \to M$. Here the tensor product is taken in a graded sense. If $a \in A_i$ and $m \in M_j$ then $a \otimes m \in (A \otimes M)_{i+j}$. 
\end{defn}

\begin{defn}
  Let $M$ be a graded module, then we define the \textbf{shift} to be 
  \begin{equation}
  (\Sigma M)_n = M_{n-1} \qquad (\Sigma^j M)_n = M_{n-j}
  \end{equation}
\end{defn}

Given this shift functor, we can define a \textit{graded} Hom functor.

\begin{defn}
  \textbf{Graded hom} between $A$-modules $M$ and $N$ is defined to be
  \begin{equation}
  \hom^n_A (M, N) = \mbf{Mod}_A (M, \Sigma^n N)
  \end{equation} 
\end{defn}

When doing homological algebra, the gradings can turn into somewhat of an annoyance. In standard homological algebra, one must deal with resolutions, and so we must keep track of how far out on a resolution we are. With graded homological algebra, we also have to keep track of gradings.

\begin{defn}
  Let $M$ be a module over a graded algebra $A_\bullet$. $M$ is \textbf{projective} if for any surjective map of $A$-modules $P \to Q$, $\hom(M,P) \to \hom(M,Q)$ is also surjective. That is, the indicated lift exists in the diagram below
  \[
  \xymatrix{
     & M \ar[d]\ar@{.>}[dl] \\
    P \ar[r] & Q \ar[r] & 0 
  }
  \]
\end{defn}

\begin{rmk}
Free modules are clearly projective, and most resolutions we deal with will in fact be free. 
\end{rmk}

\begin{defn}
  A \textbf{projective resolution} of a (graded) module $M$ is an exact sequence of (graded) modules
  \begin{equation}
  \to P_3 \to P_2 \to P_1 \to P_0 \to M
  \end{equation}
  where each $P_i$ is projective. 
\end{defn}

\begin{rmk}
Since each of the $P_i$ is graded, there are in fact \textit{two} gradings floating around in the resolution of a graded module. One is the degree of the resolution, e.g. $\deg(P_1) = 1$. The other is the grading internal to $P_i$. 
\end{rmk}

\begin{defn}
A projective resolution $P_\bullet = \cdots P_{i} \xrightarrow{f_i} P_{i-1} \cdots $ of $M$ is \textbf{minimal} if $f_i (P_i) \subset I(A) \cdot P_{i-1}$ where $I(A)$ is the kernel of the augmentation. 
\end{defn}

The key point about minimal resolutions is the following fact: if $P_\bullet \to M$ is a minimal resolution of $M$ then
\begin{equation}
\ext^{n}_A (M, k) \cong \hom_{A} (P_n, k)
\end{equation}

The Ext functors are the derived functors of $\hom^n_A (-, N)$:

\begin{defn}
  Let $A$ be a graded ring and $M, N$ graded modules. Consider a projective resolution
  \begin{equation}
 \cdots \to  P_2 \to P_1 \to P_0 \to M \to 0.
 \end{equation}
 Upon applying $\hom^t_A (-, N)$ we obtain a complex
 \begin{equation}
 \hom^t_A (P_0, N) \to \hom^t_A (P_1, N) \to \hom^t_A (P_2, N) \to \cdots .
 \end{equation} 
 Then we define $\ext^{s,t}_A (M, N)$ to be the cohomology of this complex. 
\end{defn}

\begin{rmk}
The grading $s$ is the \textit{homological grading} and $t$ is the \textit{internal grading}. 
\end{rmk}

\begin{rmk}
  An element of $\ext^{s,t}_A (M, N)$ can be represented by a homomorphism $\Sigma^t P_s \to N$: in homological degree $t$ and internal degree $s$ we are looking at the portion of a chain complex
  \begin{equation}
  \xymatrix{
    \hom^t_A (P_{s-1}, N) \ar[r]\ar@{=}[d] &  \hom^{t}_A (P_s, N) \ar[r]\ar@{=}[d] &  \hom^t_A (P_{s+1}, N)\ar@{=}[d]\\
    \hom_A (\Sigma^t P_{s-1} , N) \ar[r] & \hom_A (\Sigma^t P_s, N) \ar[r] & \hom_A (\Sigma^t P_{s+1}, N)
  }
  \end{equation}
  An element in $\ext^{s,t}_A (M, N)$ is thus a map $\Sigma^t P_s \to N$ such that $\Sigma^t P_{s+1} \to \Sigma^t P_s \to N$ is 0, i.e. $\Sigma^t \coker(P_{s+1} \to P_s) \to N$ and does not factor through $P_{s-1}$. 
\end{rmk}

\begin{thm}\label{change_of_rings}
Let $A, B$ be $k$-algebras with $B\subset A$ and $A$ flat as a $B$-module. Let $M, N$ be $A$-modules. Then
  \begin{equation}
  \ext^{s,t}_A (A \otimes_B M, N) \cong \ext^{s,t}_B (M, N)
  \end{equation}
\end{thm}

The following definition will be useful (and used frequently) in what follows. 

\begin{defn}\label{double_slash}
  Let $A$ be an algebra and $B \subset A$ an augmented subalgebra. Then define
  \begin{equation}
  A\sslash B = A \otimes_B \mbf{F}_2 \cong A / A \cdot I(B)
  \end{equation} 
  where $I(B)$ is the augmentation ideal of $B$. 
\end{defn}

\begin{example}
  This definition, combined with the change of rings, gives a useful relationship:
  \begin{equation}
  \ext^{s,t}_{A} (A \sslash B, N) \cong \ext^{s,t}_B (\mbf{F}_2, N). 
  \end{equation} 
  
\end{example}

The $\ext^{\ast,\ast}$ groups also come with a multiplicative structure. The construction is an exercise in homological algebra and detailed in \cite{mccleary}. 

\begin{prop}
  Let $A$ be a graded ring and $L, M, N$ be $A$-modules. Then there is a multiplication map
  \begin{equation}
  m: \ext^{s_1, t_1} (L, M) \otimes \ext^{s_2, t_2} (M, N) \to \ext^{s_1 +s_2, t_1+t_2}(L, N)
  \end{equation} 
\end{prop}

There is an issue of how to compute this at all. The appendix contains computations for the few concrete examples that we will need.

Note that given a short exact sequence $0 \to L \to M \to N \to 0$ of $A$-modules, we get two associated long exact sequences in $\ext^{\ast,\ast}_A (-,-)$:
\begin{equation} \label{LES_ext}
\xymatrix{
  \ext^{s,t}_A (X, L) \ar[r] & \ext^{s,t}_A (X, M) \ar[r] & \ext^{s,t}_A (X, N) \ar[dll]_\delta\\
  \ext^{s+1,t}_A (X, L) \ar[r] & \cdots & 
}
\end{equation}

and

\begin{equation}
\xymatrix{
  \ext^{s,t}_A (N, X) \ar[r] & \ext^{s,t}_A (M, X) \ar[r] & \ext^{s,t}_A (L, X) \ar[dll]_\delta\\
  \ext^{s+1,t}_A (N, X) \ar[r] & \cdots 
}
\end{equation}

The following theorem will also be useful for us. An explanation can be found in \cite[Prop. 9.6]{mccleary}

\begin{lem}\label{connecting_homomorphism}
Let $\alpha \in \ext^{1,\ast}_A (N, L)$ be represented by an extension $0 \to L \to M \to N \to 0$. Then the coboundary maps above are given by left and right multiplication by $\alpha$. 
\end{lem}

\section{The Adams Spectral Sequence}

Given a map of spaces $X \to Y$ there is a map of cohomology rings $H^\ast (Y) \to H^\ast (X)$. But by the naturality of the Steenrod operations, there is much more: the map $H^\ast (Y) \to H^\ast (X)$ is a map of $\mc{A}$-modules. That is, for any homotopy class of maps $X \to Y$ we get an element in $\hom_{\mc{A}} (H^\ast (Y), H^\ast (X))$. Since everything in sight is stable, we  get a class of maps $\Sigma^\infty_+ Y \to \Sigma^\infty_+ X$ to $\hom_{\mc{A}} (H^\ast (Y), H^\ast (X))$. The Adams spectral sequence says that this latter object is a good approximation to the homotopy groups, at least in a derived sense. 

There is an extensive literature on the Adams spectral sequence, as it is one of the most important tools in computational aspects of stable homotopy theory. Good references are \cite{hatcher, mccleary, ravenel_anss, ravenel, rognes}. 

\begin{thm}[Adams]
  Let $X$ be any spectrum, and let $Y$ be a spectrum such that $\pi_\ast (Y)$ is bounded below and $H_\ast (Y; \mbf{F}_2)$ has finite type. Then there is a spectral sequence converging to $[X, Y]^{\wedge}_{2,\ast}$ with $E_2$-term
  \begin{equation}
  E^{s,t}_2 = \ext^{s,t-s}_{\mc{A}} (H^\ast (Y), H^\ast (X))
  \end{equation}
  where the indexing is chosen so that the differential $d_2$ has degree $(-1,2)$. 
\end{thm}

\begin{example}
When $Y$ is the sphere spectrum, $S$, $H^\ast (S) = \mbf{F}_2$, so the $E_2$ term is $\ext^{s,t}_{\mc{A}}(H^\ast (X),\mbf{F}_2)$. 
\end{example}

\begin{example}
When $X = S$ as well, $[X, Y]_\ast$ is the graded 2-local maps $S \to S$, also known as the stable homotopy groups of sphers. Thus, as a first step to computing the stable homotopy groups of spheres is to compute the $E_2$-term: 
  \begin{equation}
  \ext^{s,t}_{\mc{A}} (\mbf{F}_2, \mbf{F}_2) \Longrightarrow \pi^S_{t-s} (S^\wedge_2)
  \end{equation}
  In order to compute this $E_2$-term we would need to compute an $\mc{A}$-resolution of $\mbf{F}_2$. Given the complexity of the Steenrod algebra, this is cumbersome. All of \cite{hatcher,mccleary,mosher_tangora,rognes} present careful computations of the resolution.
\end{example}

We avoid computing examples with the full Steenrod algebra. For our purposes, only the $\mc{A}(1)$ portion of the Steenrod algera will suffice.

\subsection{$\mc{A}(1)$-resolutions}

It is well known, see e.g.\cite[p.335]{adams}, that $H^\ast (ko) \cong \mc{A}\sslash\mc{A}(1)$. This information makes computing the $ko$-Adams spectral sequence an exercise in $\mc{A}(1)$. These computations will be useful later: the computations of Freed-Hopkins involve the $MSpin$-Adams spectral sequence, which in low degrees is exactly the $ko$-Adams spectral sequence.

  Suppose we want to compute $\pi_\ast (ko \sma X)$, the $ko$-homology of $X$. This is exactly $[S, ko \sma X]_\ast$, and as such the Adams spectral sequence gives
  \begin{equation}
  \ext^{s,t}_{\mc{A}} (H^\ast(ko \sma X), \mbf{F}_2) \Longrightarrow \pi^S_{t-s}(ko \sma X)
  \end{equation}
  Now, since $H^\ast (ko \sma X) \cong \mc{A} \otimes_{\mc{A}(1)} H^\ast (X)$, we may use the change-of-rings theorem \ref{change_of_rings} to obtain that the $E_2$ page is
  \begin{equation}
  \ext^{s,t}_{\mc{A}(1)}(H^\ast(X), \mbf{F}_2). 
  \end{equation}
  Thus, to compute $\pi_\ast (ko \sma X)$ it is necessary to know the  $\mc{A}(1)$-module structure on $H^\ast (X)$ and it is necessary to know how to resolve $\mc{A}(1)$-modules. 

  In this section, we present the resolution of some standard $\mc{A}(1)$-modules, as well as the computations of the corresponding $\ext$ groups. 

\begin{example}
  We give an extended example that will be important throughout the sequel. Suppose we start with $X = S$ so that the Adams spectral sequence is converging to $\pi^{S}_{t-s} (ko)$, the homotopy groups of $ko$.  The $E_2$-page of the spectral sequence is $\ext^{s,t}_{\mc{A}(1)} (\mbf{F}_2, \mbf{F}_2)$, and to compute this we construct an $\mc{A}(1)$-free resolution of $\mbf{F}_2$ as an $\mc{A}(1)$-module. This is pictured in Fig.\ref{fig:resolution_of_f2}. As with any resolution, the idea is to map generators to elements, find the kernel, then resolve the kernel. In \ref{fig:resolution_of_f2} the elements of the kernel at each resolution stage are indicated by empty circles. The kernels are then redrawn with filled circles (and sometimes rearranged slightly for easier comparison with $\mc{A}(1)$)

 \begin{figure}
  \begin{tikzpicture}[scale=.50]
    \fill(0,0) circle (3pt); 
    \A1 (2,0);
    \draw[->] (2,0) -- (0, 0);
    \draw[red] (2,1) circle (5pt);
    \draw[red] (2,2) circle (5pt);
    \draw[red] (2,3) circle (5pt);
    \draw[red] (4,3) circle (5pt);
    \draw[red] (4,4) circle (5pt);
    \draw[red] (4,5) circle (5pt);
    \draw[red] (4,6) circle (5pt);
    \fill[red] (5,1) circle (3pt);
    \fill[red] (5,3) circle (3pt);
    \fill[red] (5,4) circle (3pt);
    \fill[red] (7,2) circle (3pt);
    \fill[red] (7,3) circle (3pt);
    \fill[red] (7,5) circle (3pt);
    \fill[red] (7,6) circle (3pt);
    \draw[red] (5,1) .. controls (4,2) .. (5,3);
    \draw[red] (5,3) -- (5,4);
    \draw[red] (5,4) .. controls (6,4.5) and (6.5,6) .. (7,6);
    \draw[red] (7,2) -- (7,3);
    \draw[red] (7,3) .. controls (8,4) .. (7,5);
    \draw[red] (7,5) -- (7,6);
    \draw[red] (7,2) .. controls (6,2.5) and (5.5,4)  .. (5,4);
    \A1 (9,1);
    \A1 (12,2);
    \draw[blue,->] (9, 1) .. controls (7,.3) .. (5,1);
    \draw[blue,->] (9, 3) .. controls (7,2.3) .. (5,3);
    \draw[blue,->] (9, 4) .. controls (7,3.3) .. (5,4);
    \draw[blue,->] (11, 6) .. controls (9,5.3) .. (7,6);
    \draw[red] (9,2) circle (5pt);
    \draw[red] (11,4) circle (5pt);
    \draw[red] (11,5) circle (5pt);
    \draw[red] (11,7) circle (5pt);
    \draw[green,->] (12, 2) .. controls (9,1.3) .. (7,2);
    \draw[green,->] (12, 3) .. controls (9,2.3) .. (7,3);
    \draw[green,->] (12, 4) .. controls (9,3.2) .. (5,4);
    \draw[green,->] (14, 5) .. controls (10.5,4.3) .. (7,5);
    \draw[green,->] (14, 6) .. controls (10.5,5.3) .. (7,6);
    \draw[red] (12,5) circle (5pt);
    \draw[red] (14,7) circle (5pt);
    \draw[red] (14,8) circle (5pt);
    \draw[magenta] (9,4) circle (5pt);
    \draw[magenta] (12,4) circle (5pt);
    \draw[magenta] (11,6) circle (5pt);
    \draw[magenta] (14,6) circle (5pt);
    \fill[red] (16,2) circle (3pt);
    \fill[red] (16,4) circle (3pt);
    \fill[red] (16,5) circle (3pt);
    \fill[red] (16,6) circle (3pt);
    \fill[red] (16,7) circle (3pt);
    \fill[red] (18,4) circle (3pt);
    \fill[red] (18,5) circle (3pt);
    \fill[red] (18,7) circle (3pt);
    \fill[red] (18,8) circle (3pt);
    \sqtwoL (16,2,red)
    \sqtwoL (16,5,red)
    \sq1 (16,4,red)
    \sq1 (16,6,red)
    \sq1 (18,4,red)
    \sqtwoR (18,5,red)
    \sq1 (18,7,red)
    \sqtwoCR (16,6,red)
    \sqtwoCL (18,4,red)
    \A1 (19, 2)
    \A1 (22, 4)
  \end{tikzpicture}
  \caption{\label{fig:resolution_of_f2} The $\mc{A}(1)$-resolution of $\mbf{F}_2$}
 \end{figure}

  This gives us the resolution
  \begin{equation}
  \mbf{F}_2 \leftarrow \mc{A}(1) \leftarrow \Sigma \mc{A}(1) \oplus \Sigma^2 \mc{A}(1) \leftarrow \Sigma^2 \mc{A}(1) \oplus \Sigma^4 \mc{A}(1) \leftarrow \Sigma^3 \mc{A}(1) \oplus \Sigma^7 \mc{A}(1)
  \end{equation}
  In order to compute the cohomology, we apply $\hom^t_{\mc{A}(1)}(-,\mbf{F}_2)$ to this resolution. In fact, this is a minimal resolution, so $\ext^{s,t}_{\mc{A}(1)}(-,\mbf{F}_2) \cong \hom^t_{\mc{A}} (P_s, \mbf{F}_2)$. Thus, for example
  \[
  \ext^{2,4}_{\mc{A}(1)} (\mbf{F}_2,\mbf{F}_2) \cong \hom^4_{\mc{A}(1)} (\Sigma^2 \mc{A}(1) \oplus \Sigma^4 \mc{A}(1), \mbf{F}_2) \cong \mbf{F}_2
  \]

  This allows us to fill in the following diagram for $\ext^{s,t}_{\mc{A}(1)} (\mbf{F}_2,\mbf{F}_2)$ where each dot represents a copy of $\mbf{F}_2$.
  
  
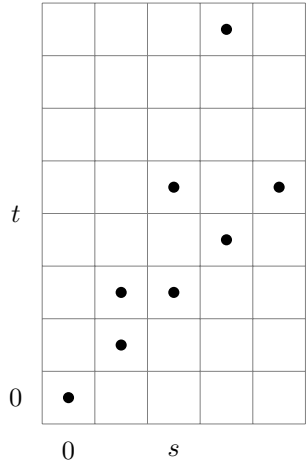
\begin{center}
  \begin{tikzpicture}[scale=.7]
    \draw[step=1cm,gray,very thin] (0,0) grid (5,8);
    \fill (.5,.5) circle (3pt);
    \fill (1.5,1.5) circle (3pt);
    \fill (1.5,2.5) circle (3pt);
    \fill (2.5,2.5) circle (3pt);
    \fill (2.5,4.5) circle (3pt);
    \fill (3.5,3.5) circle (3pt);
    \fill (3.5,7.5) circle (3pt);
    \fill (4.5,4.5) circle (3pt);
    \draw (2.5,-.5) node{$s$};
    \draw (-.5,4) node{$t$};
    \draw (.5,-.5) node{$0$};
    \draw (-.5,.5) node{$0$};
  \end{tikzpicture}
  \end{center}

  However it is standard to re-index so that the axes are $t-s$ and $s$:
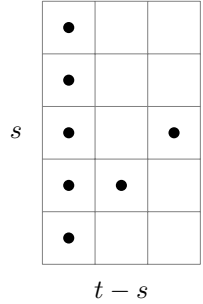
  \begin{center}
    \begin{tikzpicture}[scale =.7]
      \draw[step=1cm,gray,very thin] (0,0) grid (3,5);
      \fill (.5,.5) circle (3pt);
      \fill (.5,1.5) circle (3pt);
      \fill (.5,2.5) circle (3pt);
      \fill (.5,3.5) circle (3pt);
      \fill (.5,4.5) circle (3pt);
      \fill (1.5,1.5) circle (3pt);
      \fill (2.5,2.5) circle (3pt);
      \draw (1.5,-.5) node{$t-s$};
      \draw (-.5,2.5) node{$s$};
    \end{tikzpicture}
  \end{center}

  More structure may be teased out of the $\ext$ groups. First, the resolution is actually periodic: there is an isomoprhism
  \begin{equation}
  \ext^{s,t}_{\mc{A}(1)} (\mbf{F}_2, \mbf{F}_2) \xrightarrow{\cong} \ext^{s+4,t+12} (\mbf{F}_2,\mbf{F}_2)
  \end{equation}
  with the isomorphism given by multiplication by an element $w \in \ext^{4,12}(\mbf{F}_2,\mbf{F}_2)$. That the isomorphism is given by multiplication is far from obvious, but the periodicty can be seen by just continuing to draw out resolutions. For a proof of the multiplicative isomorphism see \cite{adams_periodicity}, \cite[Lem. 9.50]{mccleary} or \cite{rognes}.

  There is further multiplicative structure. Let $h_0$ be the generator of $\ext^{1,1}_{\mc{A}(1)} (\mbf{F}_2, \mbf{F}_2)$, $h_1$ be the generator of $\ext^{1,2}_{\mc{A}(1)} (\mbf{F}_2,\mbf{F}_2)$ and $v$ be the generator of $\ext^{3,7}_{\mc{A}(1)} (\mbf{F}_2,\mbf{F}_2)$, and $w$ be the generator of $\ext^{4,12}_{\mc{A}(1)} (\mbf{F}_2, \mbf{F}_2)$. Then we have the following theorem.

  \begin{thm}
    There is an isomorphism of rings
    \begin{equation}
    \ext^{\ast,\ast}_{\mc{A}(1)} (\mbf{F}_2,\mbf{F}_2) \cong \Z/2[h_0,h_1,v,w]/(h_0h_1, h^3_1, h_1 v, v^2 = h^2_0 w)
    \end{equation}
  \end{thm}

  The situation is represented pictorially in Fig. \ref{fig:ko_e2_page}. Each $\bullet$ in a given bidegree is a $\Z/2$. Each vertical line is multiplication by $h_0$, each diagonal line is multiplication by $h_1$.
  \begin{figure}
  \begin{center}
    \begin{tikzpicture}[scale=.7]
      \draw[step=1cm,gray,very thin] (0,0) grid (10,6);
      \fill (.5,.5) circle (3pt);
      \fill (.5,1.5) circle (3pt);
      \fill (.5,2.5) circle (3pt);
      \fill (.5,3.5) circle (3pt);
      \fill (.5,4.5) circle (3pt);
      \fill (.5,5.5) circle (3pt);
      \fill (1.5,1.5) circle (3pt);
      \fill (2.5,2.5) circle (3pt);
      \fill (4.5,3.5) circle (3pt);
      \fill (4.5,4.5) circle (3pt);
      \fill (4.5,5.5) circle (3pt);
      \fill (8.5,4.5) circle (3pt);
      \fill (8.5,5.5) circle (3pt);
      \fill (9.5,5.5) circle (3pt);
      \foreach \y in {0,1,2,3,4}
               {\draw (.5,\y+.5) -- (.5,\y+1.5);}
               \draw (.5,.5) -- (1.5,1.5);
               \draw (1.5,1.5) -- (2.5,2.5);
               \draw (4.5,3.5) -- (4.5,4.5);
               \draw (4.5,4.5) -- (4.5,5.5);
               \draw (8.5,4.5) -- (8.5,5.5);
               \draw (8.5,4.5) -- (9.5,5.5);
               \draw[->,red] (1.5,1.5) -- (.5,3.5);
               \draw[->,red] (1.5,1.5) -- (.5,4.5);
               \draw[->,red] (1.5,1.5) -- (.5,5.5);
    \end{tikzpicture}
  \end{center}
  \caption{\label{fig:ko_e2_page} The $E_2$-page of $\ext^{\ast,\ast}_{\mc{A}(1)}(\mbf{F}_2,\mbf{F}_2)$. The red arrows are possible differentials $d_2,d_3,d_4$}
  \end{figure}

  Now that we have the $E_2$-page we are in a position to figure out what the spectral sequence is converging to. First, this involves a computation of the $E_\infty$-page, which involves a computation of each of the $E_r$-pages. Luckily, the spectral sequence collapses. The only differential the spectral sequence could support is $d_r (h_1) = h_0^{r+1}$. However, suppose this were true. Then we use the fact that $d_r$ is a derivation to compute
  \begin{equation}
  d_r (h_0 h_1) = d_r (h_0) h_1 + h_0 d_r (h_1) = h^{r+2}_0
  \end{equation}
  where the last equality follows by the fact that $d_r (h_0) = 0$ for degree reasons. Now, $h_0 h_1 = 0$ in $\ext^{\ast,\ast}_{\mc{A}(1)} (\mbf{F}_2, \mbf{F}_2)$, so the above implies $0 = d_r (0) = h^{r+2}_0$ which is a contradiction, since $h_0$ generates a polynomial algebra in $\ext^{\ast,\ast}_{\mc{A}(1)} (\mbf{F}_2,\mbf{F}_2)$. These considerations show that $E_2 = E_\infty$.

  Given the $E_\infty$-page, and the indexing convention, the homotopy groups of $ko$ lie in the vertical slices of the diagram. The only thing left is to solve the extension problems --- but the multiplicative structure persists to the $E_\infty$-page. Thus, we read off the familiar (2-complete) homotopy groups of $ko$:
  \begin{align*}
    \pi_0 ko &= \Z_2 \\
    \pi_1 ko &= \Z/2 \\
    \pi_2 ko &= \Z/2 \\
    \pi_3 ko &= 0 \\
    \pi_4 ko &= \Z_2 \\
    \pi_5 ko &= 0  \\
    \pi_6 ko &= 0\\
    \pi_7 ko &= 0\\
  \end{align*}

\end{example}

\begin{example}
  Consider the $\mc{A}(1)$-module pictured at left in Fig. \ref{fig:M}, which we'll call $M$ in this example. We resolve it by first mapping a copy of $\mc{A}(1)$ onto it. We see that the kernel is simply a shifted copy of the original module, i.e. $\Sigma^1 M$. The resolution then continues in this manner.

  \begin{figure}
  \begin{center}
  \begin{tikzpicture}[scale = .5]
    \fill (0,0) circle (3pt);
    \fill (0,2) circle (3pt);
    \fill (0,3) circle (3pt);
    \fill (0,5) circle (3pt);
    \sqtwoL (0,0,black)
    \sq1 (0,2,black)
    \sqtwoL (0,3,black)
    \A1 (2, 0)
    \draw[->,blue] (2, 0) -- (0, 0);
    \draw[->,blue] (2, 2) -- (0, 2);
    \draw[->,blue] (2, 3) -- (0, 3);
    \draw[->,blue] (4, 5) -- (0, 5);
    \draw[red] (2,1) circle (5pt);
    \draw[red] (4,3) circle (5pt);
    \draw[red] (4,4) circle (5pt);
    \draw[red] (4,6) circle (5pt);
    \fill[red] (7,1) circle (3pt);
    \fill[red] (7,3) circle (3pt);
    \fill[red] (7,4) circle (3pt);
    \fill[red] (7,6) circle (3pt);
    \sqtwoL (7,1,red)
    \sq1 (7,3,red)
    \sqtwoL (7,4,red)
  \end{tikzpicture}
  \caption{\label{fig:M}Resolving the $\mc{A}(1)$-module $M$}
  \end{center}
  \end{figure}

Thus, the resolution is
\begin{equation}
M \leftarrow \mc{A}(1) \leftarrow \Sigma^1 \mc{A}(1) \leftarrow \Sigma^2 \mc{A}(1) \leftarrow \cdots 
\end{equation}

We can picture the groups $\ext^{s,t}_{\mc{A}(1)} (M, \mbf{F}_2)$ (using the $(t-s,s)$ indexing convention) as in Fig. \ref{fig:resolution_of_M}. 
\begin{figure}
\begin{center}
  \begin{tikzpicture}[scale=.7]
    \draw[step=1cm,gray,very thin] (0,0) grid (10,6);
    \foreach \y in {0,1,2,3,4,5}
             {\fill (.5,\y+.5) circle (3pt);}
             \foreach \x in {0,1,2,3,4}
             {\draw (.5,\x+.5) -- (.5,\x+1.5);}
  \end{tikzpicture}
\end{center}
\caption{\label{fig:resolution_of_M} The $E_2$-page $\ext^{s,t}_{\mc{A}(1)} (M, \mbf{F}_2)$}
\end{figure}

Here the vertical lines are indicating that the multiplication on the right by $h_0 \in \ext^{1,1}_{\mc{A}(1)} (\mbf{F}_2, \mbf{F}_2)$ takes generators to generators.

  We note that $M \cong \mc{A}(1)\sslash \mc{A}(0)$ so that, by change of rings, we could have also computed
  \[
  \ext^{s,t}_{\mc{A}(1)} (M, \mbf{F}_2) \cong \ext^{s,t}_{\mc{A}(0)} (\mbf{F}_2, \mbf{F}_2)
  \]
\end{example}

\begin{example}
  We are now in a position to write down the resolution for Joker, see Fig. \ref{fig:resolution_of_joker}. 
\begin{figure}
  \begin{center}
  \begin{tikzpicture}[scale = .4]
    \joker (0,0)
    \A1 (2,0)
    \draw[->,blue] (2,0) -- (0, 0);
    \draw[->,blue] (2,1) -- (0, 1);
    \draw[->,blue] (2,2) -- (0,2);
    \draw[->,blue] (4,3) .. controls (2, 2.3) .. (0,3);
    \draw[->,blue] (4,4) -- (0,4);
    \draw[red] (2,3) circle (5pt);
    \draw[red] (4,5) circle (5pt);
    \draw[red] (4,6) circle (5pt);
    \fill[red] (6,3) circle (3pt);
    \fill[red] (6,5) circle (3pt);
    \fill[red] (6,6) circle (3pt);
    \sqtwoL (6,3,red)
    \sq1 (6,5,red)
    \A1 (8, 3)
    \draw[->,blue] (8,3) -- (6,3);
    \draw[->,blue] (8,5) -- (6,5);
    \draw[->,blue] (8,6) -- (6,6);
    \draw[red] (8,4) circle (5pt);
    \draw[red] (10,6) circle (5pt);
    \draw[red] (10,7) circle (5pt);
    \draw[red] (10,8) circle (5pt);
    \draw[red] (10,9) circle (5pt);
    \fill[red] (12, 4) circle (3pt);
    \fill[red] (12, 6) circle (3pt);
    \fill[red] (12, 7) circle (3pt);
    \fill[red] (12, 8) circle (3pt);
    \fill[red] (12, 9) circle (3pt);
    \sqtwoR (12,4,red)
    \sq1 (12,6,red)
    \sqtwoR (12,7,red)
    \sq1 (12,8,red)
    \A1 (14,4)
    \A1 (18,8)
    \draw[->,blue] (14,4) -- (12,4);
    \draw[->,blue] (14,6) -- (12,6);
    \draw[->,blue] (14,7) -- (12,7);
    \draw[->,blue] (16,9) -- (12,9);
    \draw[->,green] (18,8) .. controls (15,7.3) .. (12,8);
    \draw[red] (14,5) circle (5pt);
    \draw[red] (16,7) circle (5pt);
    \draw[red] (16,8) circle (5pt);
    \draw[red] (16,10) circle (5pt);
    \draw[red] (18,9) circle (5pt);
    \draw[red] (18,10) circle (5pt);
    \draw[red] (18,11) circle (5pt);
    \draw[red] (20,11) circle (5pt);
    \draw[red] (20,12) circle (5pt);
    \draw[red] (20,13) circle (5pt);
    \draw[red] (20,14) circle (5pt);
    \fill[red] (22,5) circle (3pt);
    \fill[red] (22,7) circle (3pt);
    \fill[red] (22,8) circle (3pt);
    \fill[red] (22,10) circle (3pt);
    \sqtwoR (22,5,red)
    \sq1 (22,7,red)
    \sqtwoR (22,8,red)
    \fill[red] (24,9) circle (3pt);
    \fill[red] (24,11) circle (3pt);
    \fill[red] (24,12) circle (3pt);
    \fill[red] (26, 10) circle (3pt);
    \fill[red] (26, 11) circle (3pt);
    \fill[red] (26,13) circle (3pt);
    \fill[red] (26,14) circle (3pt);
    \sqtwoL (24,9,red)
    \sq1 (24,11,red)
    \sq1 (26,10,red)
    \sqtwoR (26,11,red)
    \sq1 (26,13,red)
    \sqtwoCR (24,12,red)
    \sqtwoCL (26,10,red)
  \end{tikzpicture}
  \end{center}
  \caption{\label{fig:resolution_of_joker} The $\mc{A}(1)$-resolution of the joker}
\end{figure}
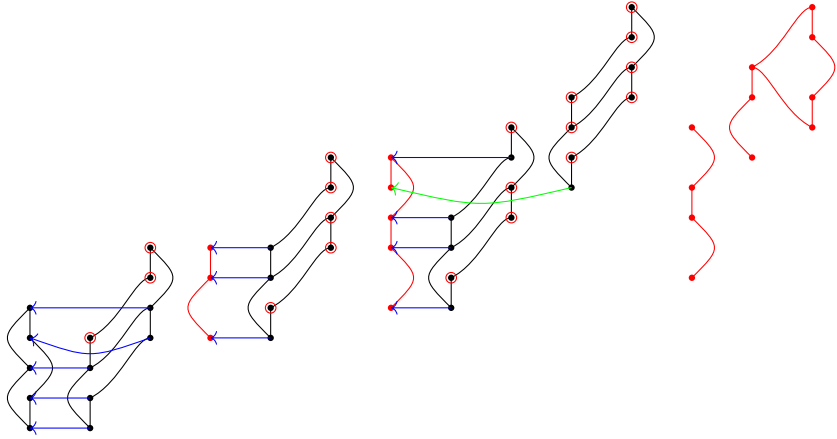
  Note that the kernel in the last step we've drawn is $\Sigma^5 M \oplus \Sigma^{10} \ker(\epsilon)$ so that the remainder of the resolution is a sum of the resolution of $\mbf{F}_2$ and $M$. That is, we have
  \begin{equation}
  J \leftarrow \mc{A}(1) \leftarrow \Sigma^3 \mc{A}(1) \leftarrow \Sigma^4 \mc{A}(1) \oplus \Sigma^8 \mc{A}(1) \leftarrow \cdots 
  \end{equation}
  The $E_2$-page is shown $\ext^{s,t}_{\mc{A}(1)} (J, \mbf{F}_2)$ in Fig. \ref{fig:e2_page_joker}. 
  \begin{figure}
  \begin{center}
  \begin{tikzpicture}[scale=.5]
    \draw[step=1cm,gray,very thin] (0,0) grid (14,10);
    \fill (.5,.5) circle (3pt);
    \fill (2.5,1.5) circle (3pt);
    \foreach \x in {2,3,4,5,6,7,8,9}
             {\fill (2.5,\x+.5) circle (3pt);}
             \foreach \y in {1,2,3,4,5,6,7,8}
                      {\draw (2.5,\y+.5) -- (2.5,\y+1.5);}
                      \foreach \y in {2,3,4,5,6,7,8,9}
                      {\fill (6.5, \y + .5) circle (3pt);}
                      \fill (7.5, 3.5) circle (3pt);
                      \fill (8.5, 4.5) circle (3pt);
                      \foreach \y in {2,3,4,5,6,7,8}
                               {\draw (6.5,\y+.5) -- (6.5, \y+1.5);}
                               \draw (6.5,2.5) -- (7.5,3.5);
                               \draw (7.5,3.5) -- (8.5,4.5);
                               
  \end{tikzpicture}
  \end{center}
  \caption{\label{fig:e2_page_joker} The $E_2$-page $\ext^{s,t}_{\mc{A}(1)} (J, \mbf{F}_2)$}
  \end{figure}
\end{example}

\begin{example}\label{periodic_reolution}
  In addition to the module $M$ above, there is another $\mc{A}(1)$-module with a useful periodic resolution, which we will call $P$ --- pictured in \ref{fig:P}. 
  \begin{figure}
\begin{center}
    \begin{tikzpicture}[scale=.50]
      \foreach \y in {0, 2, 3, 4, 5, 6, 7, 8,9,10}
               {\fill (0,\y) circle (3pt);}
               \draw (0,0) .. controls (-1, 1) .. (0, 2);
               \draw (0, 2) -- (0, 3);
               \draw (0, 3) .. controls (-1,4) .. (0, 5);
               \draw (0, 4) -- (0, 5);
               \draw (0, 4) .. controls (1, 5) .. (0, 6);
               \draw (0, 6) -- (0, 7);
               \sqtwoR (0,7,black);
               \sq1 (0,8,black);
               \sqtwoL(0,8,black);
    \end{tikzpicture}
\end{center}
\caption{\label{fig:P} The $\mc{A}(1)$-module $P$}
\end{figure}

  And we can write down a resolution (Fig. \ref{fig:resolution_of_P})
  \begin{figure}
\begin{center}
    \begin{tikzpicture}[scale = .50]
      \foreach \y in {0, 2, 3, 4, 5, 6, 7, 8,9,10}
               {\fill (0,\y) circle (3pt);}
      \foreach \x in {0, 1, 2, 3}
               {\fill (2, \x) circle (3pt);}
      \foreach \x in {3, 4, 5, 6}
               {\fill (4, \x) circle (3pt);}
               \draw (0,0) .. controls (-1, 1) .. (0, 2);
               \draw (0, 2) -- (0, 3);
               \draw (0, 3) .. controls (-1,4) .. (0, 5);
               \draw (0, 4) -- (0, 5);
               \draw (0, 4) .. controls (1, 5) .. (0, 6);
               \draw (0, 6) -- (0, 7);
               \sqtwoR(0,7,black);
               \sq1 (0,8,black);
               \sqtwoL(0,8,black);
               
               \A1 (2,0)
               \A1 (2,4)
               \draw[->,blue] (2,0) -- (0,0);
               \draw[->,blue] (2,2) -- (0,2);
               \draw[->,blue] (2,3) -- (0,3);
               \draw[->,blue] (4,5) .. controls (2,4.3) .. (0,5);
               \draw[red] (2,1) circle (5pt);
               \draw[red] (4,3) circle (5pt);
               \draw[red] (4,4) circle (5pt);
               \draw[red] (4,6) circle (5pt);
               \draw[->,green] (2,4) .. controls (1,3.3) .. (0,4);
               \draw[->,green] (2,5) .. controls (1,4.3) .. (0,5);
               \draw[->,green] (2,6) .. controls (1,5.3) .. (0,6);
               \draw[->,green] (2,7) .. controls (1,6.3) .. (0,7);
               \draw[->,green] (4,9) .. controls (2,8.3) .. (0,9);
               \draw[red] (4,7) circle (5pt);
               \draw[red] (4,8) circle (5pt);
               \draw[red] (4,10) circle (5pt);
               \draw[red] (1.75,4.75) -- (4.25,4.75);
               \draw[red] (1.75,5.25) -- (4.25,5.25);
               \draw[red] (1.75,4.75) -- (1.75,5.25);
               \draw[red] (4.25, 4.75) -- (4.25,5.25);
               \foreach \y in {1,3,4,5,6,7,8,10}
                        {\fill[red] (6,\y) circle (3pt);}
                        \sqtwoL(6,1,red);
                        \sq1 (6,3,red);
                        \sqtwoL(6,4,red);
                        \sq1 (6,5,red);
                        \sqtwoR(6,5,red);
                        \sq1 (6,7,red);
                        \sqtwoR (6,8,red);
    \end{tikzpicture}
\end{center}
\caption{\label{fig:resolution_of_P} The $\mc{A}(1)$-resolution of $P$}
  \end{figure}
    Of note in this resolution is that the two copies of $\mbf{F}_2$ indicated by the red box form one element in the kernel. That is, the left copy of $\mbf{F}_2$ be generated by $a$ and the right generated by $b$. Then $a+ b$ lies in the kernel. Furthermore, $\Sq^1(a+b) = \Sq^1(b)$ and $\Sq^2(a+b) = \Sq^2(a)$. Thus, the kernel is as pictured, and is just another copy of the original module, and we can continue the resolution.  The full resolution of this module is 
    \begin{equation}
    N \leftarrow \bigoplus \Sigma^{4k} \mc{A}(1) \leftarrow \bigoplus \Sigma^{4k+1} \mc{A}(1) \leftarrow 
    \end{equation} 

    The bigraded module $\ext^{s,t}_{\mc{A}(1)} (P, \mbf{F}_2)$ with Adams spectral sequence indexing is pictured in Fig. \ref{fig:ext_P}. Again, we will not prove that the multiplicative structure is as given. 

    \begin{figure}
\begin{center}
  \begin{tikzpicture}[scale=.5]
    \draw[step=1cm,gray,very thin] (0,0) grid (14,10);
    \foreach \y in {0,1,2,3,4,5,6,7,8,9}
             {\foreach \x in {0,4,8,12}
               {\fill (\x.5,\y.5) circle (3pt);}}
             \foreach \y in {0,1,2,3,4,5,6,7,8}
                      {\foreach \x in {0,4,8,12}
                        {\draw (\x+.5,\y+.5) -- (\x +.5,\y+1.5);}}
  \end{tikzpicture}
\end{center}
\caption{\label{fig:ext_P} $\ext^{s,t}_{\mc{A}(1)} (P, \mbf{F}_2)$}
    \end{figure}
Note that this resolution (and thus the computation of $\ext$) is essentially the same as for the module $M$, but is 4-periodic. 
\end{example}

\begin{example}
  The upside-down question mark, $Q$. The upside-down question mark is the $\mc{A}(1)$-module
  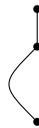
\begin{figure}
  \begin{center}
  \begin{tikzpicture}[scale=.5]
    \foreach \y in {0,2,3}
             {\fill (0,\y) circle (3pt);}
             \sqtwoL(0,0,black);
             \sq1 (0,2,black);
  \end{tikzpicture}
  \end{center}
  \caption{\label{fig:question_mark} The upside-down question mark}
  \end{figure}
  Note that this module appear as the first kernel in the resolution of the joker $J$ as an $\mc{A}(1)$-module. So, the Ext-chart is as given in Fig. \ref{fig:ext_Q}
  \begin{figure}
    \begin{center}
  \begin{tikzpicture}[scale=.5]
    \draw[step=1cm,gray,very thin] (0,0) grid (7,10);
    \foreach \x in {0,1,2,3,4,5,6,7,8,9}
             {\fill (0.5,\x+.5) circle (3pt);}
             \foreach \y in {0,1,2,3,4,5,6,7,8}
                      {\draw (0.5,\y+.5) -- (0.5,\y+1.5);}
                      \foreach \y in {0,1,2,3,4,5,6,7,8,9}
                      {\fill (4.5, \y + .5) circle (3pt);}
                      \fill (5.5, 1.5) circle (3pt);
                      \fill (6.5, 2.5) circle (3pt);
                      \foreach \y in {0,1,2,3,4,5,6,7,8}
                               {\draw (4.5,\y+.5) -- (4.5, \y+1.5);}
                               \draw (4.5,0.5) -- (5.5,1.5);
                               \draw (5.5,1.5) -- (6.5,2.5);
                               
  \end{tikzpicture}
    \end{center}
    \caption{\label{fig:ext_Q} $\ext^{s,t}_{\mc{A}(1)} (Q, \mbf{F}_2)$}
  \end{figure}
\end{example}

\section{Computations from Freed-Hopkins}\label{freed_hopkins_computations}

In this section we finally turn to the computations outlined in \cite{freed_hopkins}. For extensive comparison of these results with both the physics literature and free fermion theories, see \cite[Sec. 9.3]{freed_hopkins}. 

In \cite{freed_hopkins} the authors identify ten examples of Thom spectra that classify invertible topological phases, in the notation of their paper $MTH(d)$ with $-3 \leq d \leq 4$ \cite[Table. 9.33]{freed_hopkins} and the complex versions. The phases are identified with homotopy groups of maps into the Anderson dual of the sphere $\Sigma^{d+1} I_{\Z}$, but in order to compute these groups it is enough to compute the homotopy groups $\pi_\ast MTH(d)$. This follows from the fact that by definition of $I_{\Z}$ there is an exact sequence
\begin{equation}\label{anderson_exact_sequence}
0 \to \ext^1 (\pi_{n-1} X, \Z) \to [X, \Sigma^n I_{\Z}] \to \hom(\pi_n X, \Z) \to 0
\end{equation}
for any spectrum $X$ (\cite[Def. B.2]{hopkins_singer}). 

The first step in the computation is to identify the spaces $MTH(d)$ in terms of more familiar entities. The Freed-Hopkins (see \cite[Eqn. 10.2]{freed_hopkins}) gives the following homotopy equivalences for $MTH(d)$
\begin{align}
  & \Sigma^{-d} MSpin \sma MTO_{|d|} \qquad -3 \leq d \leq 0\\
  & \Sigma^{-d} MSPin \sma MO_{|d|} \qquad 0 \leq d \leq 3\\
  &\Sigma^{-d+1} MSpin \sma MSO_3 \qquad d = 4
\end{align}

The standard way to compute the homotopy groups of these spectra is to use the Adams spectral sequence. In, for example, the first group of spectra, the $E_2$-term is
\begin{equation}
\ext^{s,t}_{\mc{A}} (H^\ast (\Sigma^{-d} MSpin\sma MTO_{|d|}), \mbf{F}_2) 
\end{equation}
However, by the K\"{u}nneth formula for spectra
\begin{equation}
H^\ast (\Sigma^{-d} MSpin \sma MTO_{|d|}) \cong H^\ast(MSpin) \otimes H^\ast (\Sigma^{-d} MTO_{|d|}). 
\end{equation} 

It is known (see, e.g. \cite[Thm. 8.1]{anderson_brown_peterson}) that $H^\ast (MSpin) \cong \mc{A} \otimes_{\mc{A}(1)} (\Z/2 \oplus M)$ where $M$ is some graded module which is zero in degrees 0 through 7. Using this and the change of rings isomorphism for $\ext$, Lem.\ref{change_of_rings} we get that when $t-s < 8$, the $E_2$-term of the spectral sequence is
\begin{equation}
\ext^{s,t}_{\mc{A}(1)} (H^{\ast-d} MTO_{|d|}, \mbf{F}_2). 
\end{equation}

\begin{rmk}
  Another way of seeing this result is that in low degrees $\pi_\ast ko$ and $\pi_\ast MSpin$ are isomorphic with the isomorphism given by the Atiyah-Bott-Shapiro map $MSpin \to ko$, which is an isomorphism in degrees $\leq 7$. This means that computing $\pi_\ast (MSpin \sma X)$ in low degree is equivalent to computing $\pi_\ast (ko \sma X)$, and by the Adams spectral sequence, this amounts to a computation of the $\mc{A}$-action on $H^\ast (ko \sma X)$, which we know since $H^\ast (ko \sma X) \cong \mc{A} \otimes_{\mc{A}(1)} H^\ast (X)$. Again, change of rings gives the $E_2$-page of the Adams spectral sequence as above. 
\end{rmk}

We have thus reduced our work of computing the homotopy groups of $MTH(s)$ to computing the action of the Steenrod algebra on $MTO_{|d|}$, $MO_{|d|}$ and $MSO_3$, finding a resolution and doing the usual homological algebra. Luckily, we have seen in Section 2 how to compute the $\mc{A}$-action on the Thom spectra $MTO_n$ and $MO_n$: a combination of the definition of Stiefel-Whitney classes, the Wu formula, and the Cartan formula provides all of the necessary input. 

Before moving on to computations, in the first subsection we collect a few results about the $\mc{A}(1)$-action on Thom spectra and then we move on to explicitly compute the resolutions and the homotopy groups in a number of cases. 

\subsection{Action of Steenrod Algebra}
For future computations, we will need some tables of the Steenrod action on Stiefel-Whitney classes and Thom spaces. Recall the Wu formula and the definition of the Whitney class:
\begin{align}
  \Sq^i (w_j) &= \sum^i_{k=0} \binom{(j-i) + (k-1)}{k} w_{i-k} w_{j+k}\\
  \Sq^i (U) &= w_i \smile U
\end{align}

Calculating some low dimensional Steenrod squares on Stiefel-Whitney classes we obtain the table: 
\begin{center}
\begin{tabular}{|c|c|c|c|c|}\hline
 &  $\Sq^1$ & $\Sq^2$ & $\Sq^3$ & $\Sq^4$ \\\hline
$w_1$ &  $ w_1^2$ & & &  \\\hline
  $w_2$ &  $w_1 w_2 + w_3$ & $w^2_2$ & &  \\\hline
  $w_3$ & $w_1 w_3$ & $w_2 w_3 + w_1 w_4 + w_5$ & $w^3_3$ & \\\hline
  $w_4$ & $w_1 w_4 + w_5$ & $w_2 w_4 + w_6$ & $w_3 w_4 + w_2 w_5 + w_1 w_6 + w_7$ & $w^2_4$\\\hline
\end{tabular}
\end{center}

We can also tabulate Steenrod squares of various polynomials in Stiefel-Whitney classes. These are calculated using a combination of the squares above and the Cartan formula. For example
\begin{align*}
  \Sq^2 (w_2 U) &= \Sq^2 (w_2) U + \Sq^1 (w_2) \Sq^1 U + w_2 \Sq^2 U \\
  &= w^2_2 U + (w_1 w_2 + w_3)\cdot w_1 U + w^2_2 U \\
  &= (w^2_1 w_2 + w_1 w_3) U
\end{align*}

Similar computations allow us to fill in the following table (some of the computations are slighly lengthy and a good exercise): 

\begin{align*}
  \Sq^1 (w_1 U) &= 0\\
  \Sq^2 (w_1 U) &= (w^3_1 + w_1 w_2) U\\
  \Sq^1 (w_2 U) &= w_3 U\\
  \Sq^2 (w_2 U) &= (w^2_1 w_2 + w_1 w_3) U\\
  \Sq^1 (w^1_2 U) &= w^3_1 U\\
  \Sq^2 (w^1_2 U) & = (w^4_1 + w^2_1 w_2) U\\
  \Sq^2 (w_3 U) &= (w^2_1 w_3 + w_1 w_4 + w_5)U \\
  \Sq^1 (w_1 w_2 U) &=  (w^2_1 w_2 + w_1 w_3)U\\
  \Sq^2 (w_1 w_2 U) &= w^3_1 w_2 U \\
  \Sq^1 (w^1_3 U) &= 0 \\
  \Sq^2 (w^1_3 U) &= w^3_1 w_2 U
\end{align*}

\begin{rmk}
In order to fully compute the examples below, these tables would have to be carried out quite a bit further; at that point, computer computations become the better option. 
\end{rmk}

\subsection{Computations}

In this subsection we carry out the computation of $\pi_\ast MTH(s)$ for a number of examples. The procedure is always the same:

\begin{enumerate}
\item Identify $MTH(s)$ in terms of a shift of $MSpin$ and some Thom spectrum
\item Compute the $\mc{A}(1)$-action on the cohomology of the corresponding Thom spectrum. 
\item Compute the $E_2$-page of the corresonding Adams spectral sequence. This can be accomplished by either
  \begin{enumerate}
  \item Computing an explicit resolution of the $\mc{A}(1)$-module from step 2
  \item Identifying the $\mc{A}(1)$-module as living in a short exact sequence of $\mc{A}(1)$-modules where the other members of the sequence have known resolutions and invoking the long exact sequences in Ext. 
  \end{enumerate}
\item Compute the homotopy groups from the $E_\infty$-page of the spectral sequence. 
\end{enumerate}

\begin{example}
  The case $s = 1$. In this case
  \begin{equation}
  \Sigma^1 MTH(1) \simeq MSpin \sma MO(1) 
  \end{equation}
  so that $\pi_\ast MTH(1)$ is computed by computing the $MSpin$-homology of $\Sigma^{-1} MO(1)$. Note that $MTH(1)$ is $MTPin^{-} = MPin^+$ (see table 9.33 in \cite{freed_hopkins})
  
  We can explicitly compute the full $\mc{A}(1)$-module structure on $MO(1)$ using the definition of the Stiefel-Whitney class and the Wu formula. To compute the action of $\Sq^1$ on a representative element $x^n U$, we need only use the fact that $\Sq^1$ is a derivation: 
  \begin{equation}
  \Sq^1 (x^n U) = \Sq^1 (x^n) U + x^n \Sq^1 (U) = nx^{n+1} U + x^n x U = (n+1) x^{n+1} U.
  \end{equation}
 To compute the action of $\Sq^2$ we use the Cartan formula and the fact that $\Sq^2 U = 0$ in $H^\ast (MO(1))$:  
  \begin{align}
    \Sq^2 (x^n U) &= \Sq^2 (x^n) U  + \Sq^1 (x^n) \Sq^1 (U) + x^n \Sq^2 (U)\\
    &= \frac{n(n-1)}{2} x^{n+2} U + n x^{n+1} U = \frac{n(n+1)}{2} x^{n+2} U
  \end{align}
  Thus, $MO(1)$ has almost the same structure as $\R P^\infty$ --- the picture (Fig. \ref{fig:MO1}) looks the same except we're missing the 0-cell and each dot  now represents $x^{n} U$ instead of $x^n$.
  \begin{figure}
    \begin{center}
  \begin{tikzpicture}[scale=.5]
    \foreach \y in { 1, 2, 3, 4, 5, 6, 7, 8}
             {\fill (0, \y) circle (3pt);}
             \draw (0, 1) -- (0, 2);
             \draw (0, 2) .. controls (-1, 3) .. (0, 4);
             \draw (0, 3) -- (0, 4);
             \draw (0, 3) .. controls (1, 4) .. (0, 5);
             \draw (0, 5) -- (0, 6);
             \draw (0, 6) .. controls (1, 7) .. (0, 8);
             \draw (0, 7) -- (0, 8); 
  \end{tikzpicture}
    \end{center}
    \caption{\label{fig:MO1} The $\mc{A}(1)$ structure of $\widetilde{H}^\ast (MO(1))$}
  \end{figure}

  To compute the $E_2$-page of the Adams spectral sequence for this module we note that there is short exact sequence of $\mc{A}(1)$-modules:
  \begin{equation}
  \mbf{F}_2 \to P \to H^\ast (MO(1)) 
  \end{equation}
  where $P$ has the periodic resolution, Ex. \ref{convenient_resolution}. A picture of this short exact sequence is given in Fig. \ref{ses_MPin+}. 

  \begin{figure}
    \begin{center}
      \begin{tikzpicture}[scale=.5]
        \fill (0,0) circle (3pt);
      \foreach \y in {0, 2, 3, 4, 5, 6, 7, 8,9,10}
               {\fill (2,\y) circle (3pt);}
               \draw (2,0) .. controls (1, 1) .. (2, 2);
               \draw (2, 2) -- (2, 3);
               \draw (2, 3) .. controls (1,4) .. (2, 5);
               \draw (2, 4) -- (2, 5);
               \draw (2, 4) .. controls (3, 5) .. (2, 6);
               \draw (2, 6) -- (2, 7);
               \sqtwoR (2,7,black);
               \sq1 (2,8,black);
               \sqtwoL(2,8,black);
  \foreach \y in { 1, 2, 3, 4, 5, 6, 7, 8}
             {\fill (4, \y+1) circle (3pt);}
             \draw (4, 2) -- (4, 3);
             \draw (4, 3) .. controls (3, 4) .. (4, 5);
             \draw (4, 4) -- (4, 5);
             \draw (4, 4) .. controls (5, 5) .. (4, 6);
             \draw (4, 6) -- (4, 7);
             \draw (4, 7) .. controls (5, 8) .. (4, 9);
             \draw (4, 8) -- (4, 9);
             \draw[->,red] (0,0) -- (2,0);
             \draw[->,red] (2,2) -- (4,2);
      \end{tikzpicture}
      \caption{\label{ses_MPin+}The short exact sequence of $\mc{A}(1)$-modules $\mbf{F}_2 \to P \to H^\ast (MO(1))$}
    \end{center}
  \end{figure}

  This gives us a long exact sequence on $\ext$ groups Eqn.\ref{LES_ext}
  \begin{equation}
  \ext^{\ast,\ast} (H^\ast (MO), \mbf{F}_2) \to \ext^{\ast,\ast}(P, \mbf{F}_2) \to \ext^{\ast,\ast}(\mbf{F}_2, \mbf{F}_2). 
  \end{equation}

  Using the known $\ext$ groups of $P$ of $\mbf{F}_2$ we see from the long exact sequence that whenever $\delta$ is non-zero it is an isomorphism. The connecting homomorphism is a map $\ext^{s,t}\to \ext^{s+1,t}$, which upon re-indexing becomes a degree $(-1,1)$ map. Drawing both the resolutions of $P$ and of $\mbf{F}_2$ on the same figure (Fig. \ref{fig:MPin+}), the connecting homomorphism is represented by the red lines.

  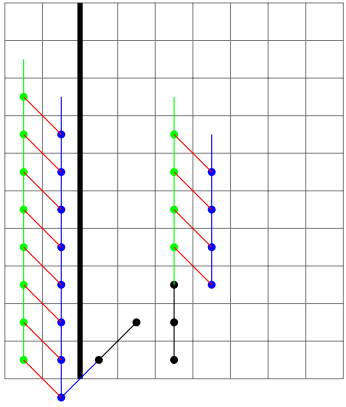
\begin{figure}
    \begin{center}
    \begin{tikzpicture}[scale=.5]
      \draw[step=1cm,gray,very thin] (-2,0) grid (7,10);
      \draw[line width=2pt] (0,0) -- (0,10);
      \foreach \y in {0, 1, 2, 3, 4, 5, 6, 7}
         {\fill[green] (-1.5,\y+.5) circle (3pt);
           \draw[green] (-1.5,\y+.5)--(-1.5,\y+1.5);}
         \foreach \y in {0,1,2,3,4,5,6,7}
                  {\fill[blue] (-.5,\y-.5) circle (3pt);
                    \draw[blue] (-.5,\y-.5) -- (-.5,\y+.5);}
       \foreach \y in {0,1,2,3,4,5,6,7}
       {\draw[red] (-.5,\y-.5) -- (-1.5,\y+.5);}
      \draw[blue] (-.5,-.5) -- (.5,.5);
      \fill (.5,.5) circle (3pt);
      \fill (1.5,1.5) circle (3pt);
      \fill (2.5,.5) circle (3pt);
      \fill (2.5,1.5) circle (3pt);
      \fill (2.5,2.5) circle (3pt);
      \draw (.5,.5) -- (1.5,1.5);
      \draw (2.5,.5) -- (2.5,1.5);
      \draw (2.5,1.5) -- (2.5,2.5);
      \draw[green] (2.5,2.5) -- (2.5,3.5);
      \foreach \y in {3,4,5,6}
      {\fill[green] (2.5,\y+.5) circle (3pt);
        \draw[green] (2.5,\y+.5) -- (2.5,\y+1.5);}
      \foreach \y in {3,4,5,6}
      {\fill[blue] (3.5,\y-.5) circle (3pt);
        \draw[blue] (3.5,\y-.5) -- (3.5,\y+.5);}
      \foreach \y in {3,4,5,6}
      {\draw[red] (3.5,\y-.5) -- (2.5,\y+.5);}
    \end{tikzpicture}
    \end{center}
    \caption{\label{fig:MPin+}The $E_2$-page of the Adams spectral sequence computing $MT\Pin^- \simeq M\Pin^+$. The bold line is $t-s = 0$.}
  \end{figure}

  Recalling that the homotopy groups lie in the columns of \label{fig:MPin+} we thus obtain

  \begin{thm}
    The low-dimensional homotopy groups of $MTH(-1) \simeq \Sigma^{-1} M\Spin \sma MO(1) \simeq M T \Pin^- \simeq M\Pin^+$ are
    \begin{align*}
      \pi_0 M\Pin^+ &= \Z/2\\
      \pi_1 M\Pin^+ &= \Z/2 \\
      \pi_2 M\Pin^+ &= \Z/8 \\
      \pi_3 M\Pin^+ &= 0\\
      \pi_4 M\Pin^+ &= 0
    \end{align*}
  \end{thm}

  \begin{rmk}
    This is, of course, well known (see, e.g. \cite{kirby_taylor}), but computed by different methods. 
  \end{rmk}

\end{example}

\begin{example}
  $s = 1$. In this case
  \begin{equation}
  \Sigma^{-1} MTH(1) \simeq MSpin \sma MTO(1)
  \end{equation}
  and  $MTH(1)$ is also equivalent to $MT\Pin^+ \simeq M\Pin^-$. 

  We have that $MTO(1) \simeq \R P^\infty_{-1}$. The corresponding $\mc{A}(1)$ module is given in Fig. \ref{fig:MTO1}. 
\begin{figure}
  \begin{center}
  \begin{tikzpicture}[scale = .5]
    \foreach \y in {0,1,2,3,4,5,6,7}
             {\fill (0,\y) circle (3pt);}
             \draw (0,0) -- (0,1);
             \draw (0,2) -- (0,3);
             \draw (0,4) -- (0,5);
             \draw (0,6) -- (0,7);
             \sqtwoL (0,0,black);
             \sqtwoL (0,3,black);
             \sqtwoR (0,4,black);
             \sqtwoR (0,7,black);
  \end{tikzpicture}
  \end{center}
  \caption{\label{fig:MTO1} The $\mc{A}(1)$-module structure of $MTO(1) \simeq \R P^\infty_{-1}$. It begins in degree $-1$}
\end{figure}
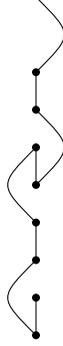

  The $\mc{A}(1)$-module $\Sigma^1 \R P^\infty_{-1}$ sits in an exact sequence of $\mc{A}(1)$-modules
  \begin{equation}
  \Sigma^1 \mbf{F}_2 \to \Sigma^1 H^\ast (\R P^\infty_{-1}) \to M 
  \end{equation}
  from which we get a long exact sequence on Ext groups
  \begin{equation}
  \ext^{s,t}(M, \mbf{F}_2) \to \ext^{s,t}(\Sigma^1 H^\ast (\R P^\infty_{-1}), \mbf{F}_2) \to \ext^{s,t} (\Sigma^1 \mbf{F}_2, \mbf{F}_2) \to \ext^{s+1,t} (M, \mbf{F}_2)
  \end{equation}
  The connecting homomorphism in this long exact sequence is given by multiplication by the element in $\ext^{1,\ast} (M, \Sigma^1 \mbf{F}_2)$ corresponding to the extension. As in the previous example, we can draw everything on the same $E_2$-page (Fig. \label{fig:MPin-}). 

  From the long exact sequence in $\ext$, we see that whenever $\delta$ is non-zero it is an isomorphism.
  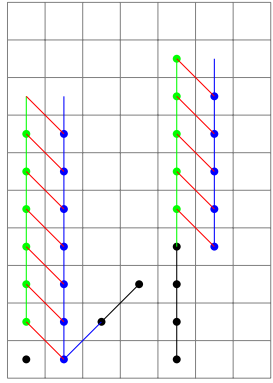
\begin{figure}
\begin{center}
  \begin{tikzpicture}[scale=.5]
    \draw[step=1cm,gray,very thin] (0,0) grid (7,10);
    \fill (0.5,0.5) circle (3pt);
    \foreach \y in {1,2,3,4,5,6}
             {\fill[green] (0.5,\y+.5) circle (3pt);
               \draw[green] (0.5,\y+.5) -- (0.5,\y+1.5);}
             \foreach \y in {0,1,2,3,4,5,6}
                      {\fill[blue] (1.5,\y+.5) circle (3pt);
                        \draw[blue] (1.5,\y+.5) -- (1.5,\y+1.5);}
                      \fill (2.5,1.5) circle (3pt);
                      \draw[blue] (1.5,.5) -- (2.5,1.5);
                      \draw (2.5,1.5) -- (3.5,2.5);
                      \fill (3.5,2.5) circle (3pt);
                      \foreach \y in {0,1,2,3}
                               {\fill (4.5,\y+.5) circle (3pt);}
                               \foreach \y in {0,1,2}
                                        {\draw (4.5,\y+.5) -- (4.5,\y+1.5);}
                               \foreach \y in {4,5,6,7,8}
                                        {\fill[green] (4.5,\y+.5) circle (3pt);
                                          \draw[green] (4.5,\y-.5) -- (4.5,\y+.5);}
                                        \foreach \y in {3,4,5,6,7}
                                                 {\fill[blue] (5.5,\y+.5) circle (3pt);
                                                   \draw[blue] (5.5,\y+.5) -- (5.5,\y+1.5);}
                                                 \foreach \y in {0,1,2,3,4,5,6}
                                                          {\draw[red] (1.5,\y+.5) -- (.5,\y+1.5);}
                                                          \foreach \y in {3,4,5,6,7}
                                                          {\draw[red] (5.5,\y+.5) -- (4.5,\y+1.5);}
                                                          
  \end{tikzpicture}
\end{center}
\caption{\label{fig:MPin-}The $E_\infty$-page computing $\pi_\ast (\Sigma^{-1} MTH(1))$}
  \end{figure}

The vertical lines at $t-s = 0, 4, \dots$ are the bigraded $\ext$ groups of $M$. The rest of the diagram corresponds to the Ext groups of $\Sigma^1 H^\ast (\R P^\infty_{-1})$. The red lines indicate the connecting homomorphism. The black elements indicate the cokernel of $\delta$, i.e. everything left over.

Reading off the homotopy groups, we obtain the following

\begin{thm}
  The homotopy groups of $MTH(-1) \simeq MT\Pin^+ \simeq M\Pin^-$ are given by
  \begin{align*}
    \pi_0 M\Pin^- &= \Z/2 \\
    \pi_1 M\Pin^- &= 0 \\
    \pi_2 M\Pin^- &=\Z/2 \\
    \pi_3 M\Pin^- &= \Z/2\\
    \pi_4 M\Pin^= &= \Z/16
  \end{align*}
\end{thm}
\end{example}

\begin{example}
  The case $s = 2$. This case corresponds to the homotopy groups of $MT\Pin^{\tilde{c}-}$ and we have the equialence
  \[
  MT\Pin^{\tilde{c}-} \simeq M \Spin \sma S^{-1} \sma MO(2).
  \]
  Since everything is stable, it suffices to compute $H^\ast (MO(2))$ as an $\mc{A}(1)$-module. Recall that $H^\ast (BO(2)) \cong \Z/2[w_1, w_2]$ and thus, by the Thom isomorphism $H^\ast (MO(2))$ is, as a vector space
  \begin{equation}
  H^\ast (MO(2)) \cong \Z/2[ w_1 U, w_2 U, w^2_1 U, w^3_1 U, w_1 w_2 U, w^4_1 U, ...]
  \end{equation} 
  We now use our previously computed Steenrod actions. For example, we can look at the action of $\mc{A}(1)$ on $U$, the generator in degree 0. 
  \begin{align*}
    \Sq^1 U &= w_1 U\\
    \Sq^2 U &= w_2 U \\
    \Sq^2 \Sq^1 U &= \Sq^2 (w_1 U) = (w^3_1 + w_1 w_2) U\\
    \Sq^3 U &= w_3 U = 0 \\
    \Sq^3 \Sq^1 U &= \Sq^2 \Sq^2 U =  \Sq^3 (w_1 U) = \Sq^3 w_1 U + \Sq^2 w_1 \Sq^1 U + \Sq^1 w_1 \Sq^2 U + w_1 \Sq^3 U \\
    &= w^2_1 w_2 U\\
    \Sq^5U + \Sq^4\Sq^1 U &= \Sq^2 \Sq^3 U = 0\\
    \Sq^5 \Sq^1 U &= \Sq^1 (\Sq^5 U + \Sq^4 \Sq^1 U) = 0
  \end{align*}
  Thus, we see that the action of $\mc{A}(1)$ on $U$ generates a joker:
  
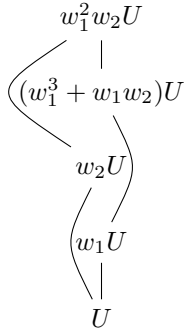
\begin{center}
  \begin{tikzpicture}[scale=.5]
    \draw (0,0) node(0){$U$};
    \draw (0,2) node(1){$w_1 U$};
    \draw (0,4) node(2){$w_2 U$};
    \draw (0,6) node(3){$(w^3_1 + w_1 w_2)U$};
    \draw (0,8) node(4){$w^2_1 w_2 U$};
    \draw (0) -- (1);
    \draw (3) -- (4);
    \draw (0) .. controls (-1,2) .. (2);
    \draw (2) .. controls (-3,6) .. (4);
    \draw (1) .. controls (1, 4) .. (3); 
  \end{tikzpicture}
  \end{center}

  In the diagram below, each row represents an element in the corresponding degree --- the bottommost dot corresponds to $U$ in degree 0. Above that is $w_1 U$ and the next highest dots are $w^2_1, w_2$, etc. The joker above corresponds to the lower-left joker in the figure.

  The element in lowest degrees not generated by the action of $\mc{A}(1)$ on $U$ is $w^2_1 U$. One then starts the process again --- we see the $\mc{A}(1)$-action on $w^2_1 U$. Through much computation one sees that it's a full $\mc{A}(1)$. One continues this process to get the structure in Fig.\ref{fig:MO2}. 

  \begin{figure}
  \begin{center}
  \begin{tikzpicture}[scale = .50]
    \joker (0, 0);
    \joker (0, 8);
    \fill (0, 6) circle (3pt);
    \A1 (2, 8);
    \A1 (3, 6);
    \A1 (4, 4);
    \A1 (5, 2)
  \end{tikzpicture}
  \end{center}
  \caption{\label{fig:MO2} The $\mc{A}(1)$-module structure of $H^\ast (S^{-1} \sma MO(2))$}
  \end{figure}

  So, in low degrees, $S^{-1} \sma MO(2)$ is, as an $\mc{A}(1)$-module
  \begin{align}
    H^\ast(S^{-1} \sma MO(2)) &\cong J \vee \Sigma^8 J \vee \Sigma^5 \mbf{F}_2 \vee \Sigma^2 \mc{A}(1) \vee \Sigma^3 \mc{A}(1)\\
    &\vee \Sigma^4 \mc{A}(1) \vee \Sigma^5 \mc{A}(1) \vee \Sigma^6 \mc{A}(1) \vee \Sigma^7 \mc{A}(1) \vee \Sigma^8 \mc{A}(1)
  \end{align}
  Since we know the resolutions of $\mbf{F}_2, J$ and $\mc{A}(1)$ as $\mc{A}(1)$-modules, we can easily compute $\ext^{s,t}_{\mc{A}(1)} (H^\ast (S^{-1} \sma MO(2)), \mbf{F}_2)$. The table, with the usual Adams indexing is Fig. \ref{fig:e2_page_MO2}  (in the degrees which will matter for us)
\begin{figure}
  \begin{center}
    \begin{tikzpicture}[scale=.5]
      \draw[step=1cm,gray,very thin] (0,0) grid (14,10);
      \fill (0.5,0.5) circle (3pt);
      \fill (2.5,0.5) circle (3pt);
      \fill (4.5,0.5) circle (3pt);
      \fill (6.25,0.5) circle (3pt);
      \fill (6.75,.5) circle (3pt);
      \fill (8.25,.5) circle (3pt);
      \fill (8.75,.5) circle (3pt);
      \draw (6.75,.5) -- (6.75,1.5);
      \fill (7.25,4.5) circle (3pt);
      \fill (8.25,5.5) circle (3pt);
      \draw (6.25,3.5) -- (7.25,4.5);
      \draw (7.25,4.5) -- (8.25,5.5);
      \fill (7.75,1.5) circle (3pt);
      \fill (8.75,2.5) circle (3pt);
      \draw (6.75,.5) -- (7.75,1.5);
      \draw (7.75,1.5)-- (8.75,2.5);
      \foreach \y in {1,2,3,4,5,6,7,8}
              {\fill (2.5,\y+.5) circle (3pt);
                \draw (2.5,\y+.5) -- (2.5, \y+1.5);
                \fill (6.75, \y+.5) circle (3pt);
                \draw (6.75,\y+.5) -- (6.75,\y+1.5);}
              \foreach \x in {3,4,5,6,7,8}
                       {\fill (6.25,\x+.5) circle (3pt);
                         \draw (6.25,\x+.5) -- (6.25,\x+1.5);}
    \end{tikzpicture}
  \end{center}
  \caption{\label{fig:e2_page_MO2} The $E_2 = E_\infty$ page of $\pi_\ast MTPin^{\tilde{c}-}$}
    \end{figure}

From this, we can read off the low-dimensional homotopy groups of $MT\text{Pin}^{\tilde{c}-}$:

\begin{thm}
  The low dimensional homotopy grous of $MT\Pin^{\tilde{c}-}$ are 
  \begin{align}
    \pi_0 MT\text{Pin}^{\tilde{c}-} &= \Z/2\Z \\
    \pi_1 MT\text{Pin}^{\tilde{c}-} &= 0\\
    \pi_2 MT\text{Pin}^{\tilde{c}-} &= \Z \times \Z/2\Z\\
    \pi_3 MT\text{Pin}^{\tilde{c}-} &= 0 \\
    \pi_4 MT\text{Pin}^{\tilde{c}-} &= \Z/2\Z
  \end{align}
  \end{thm}
\end{example}

\begin{example}
  The case $s = -2$. This corresponds to $MT \Pin^{\tilde{c}+}$, and we have the identification
  \begin{equation}
    MT\Pin^{\tilde{c}+} \simeq M\Spin \sma S^{2} \sma MTO(2).
    \end{equation}

    In this case we have to use our knowledge of the $\mc{A}(1)$ action on $MTO(2)$. Let $U^\perp$ be the Thom class for $\gamma^\perp \to BO(2)$. Recall that
  \begin{equation}
  \Sq^1(U^\perp) = w_1 U^\perp \qquad \Sq^2 (U^\perp) = (w^2_1 + w_2) U^\perp
  \end{equation}

  We proceed as in the other examples, but with this modified $\mc{A}(1)$-structure. We can compute that $U^{\perp}$ generates a full $\mc{A}(1)$. The lowest degree element of the $\Z/2$-basis for $H^\ast (MTO(2))$ that is not picked up in this process is $w_2 U^\perp$. We compute
  \begin{equation}
    \Sq^1 (w_2 U^\perp) = \Sq^1(w_2) U^\perp + w_2 U^\perp = w_1 w_2 U^\perp + w_1 w_2 U^\perp = 0
    \end{equation}
  where the second equality follows from the fact that $\Sq^1 (w_2) = w_1 w_2 + w_3$, but $w_3$ vanishes in this case. Also,
  \begin{align}
    \Sq^2(w_2 U^\perp) &= \Sq^2(w_2) U^\perp + \Sq^1(w_2) \Sq^1(U^\perp) + w_2 \Sq^2(U^\perp)\\
    &= w^2_2 U^\perp + (w_1 w_2 + w_3)w_1 U^\perp + w_2 (w^2_1 + w_2)U^\perp \\
    &= 0 
  \end{align}

  So $w_2 U^\perp$ generates an $\mbf{F}_2$. We obtain the  $\mc{A}(1)$-structure for $H^\ast (S^{-2} \sma MTO(2))$ pictured in Fig. \ref{fig:MTO2}. 

  \begin{figure}
  \begin{center}
    \begin{tikzpicture}[scale=.5]
      \fill (0,2) circle (3pt);
      \joker(0,4);
      \A1 (2,0);
      \A1 (2,4);
      \A1 (2,8);
      \A1 (7,6);
      \A1 (6,8);
    \end{tikzpicture}
  \end{center}
  \caption{\label{fig:MTO2} The $\mc{A}(1)$-module structure on $S^{-2} \sma MTO(2)$}
  \end{figure}

  Besides the copies of $\mc{A}(1)$, the resolution of $H^\ast(S^{-2} \sma MTO(2))$ will contain a copy of the resolution of $J$ and of $\mbf{F}_2$. This allows for computation of the $E_2$-page (\ref{fig:e2_page_MTO2}. 

  \begin{figure}
  \begin{center}
  \begin{tikzpicture}[scale=.5]
    \draw[step=1cm,gray,very thin] (0,0) grid (14,10);
    \fill (.5,.5) circle (3pt);
    \fill (4.25,.5) circle (3pt);
    \fill (4.75,.5) circle (3pt);
    \fill (6.5,.5) circle (3pt);
    \fill (8.25,.5) circle (3pt);
    \fill (8.75,.5) circle (3pt);
    \fill (3.5,1.5) circle (3pt);
    \fill (4.5,2.5) circle (3pt);
    \draw (2.5,.5) -- (3.5,1.5);
    \draw (3.5,1.5) -- (4.5,2.5);
    \foreach \y in {0,1,2,3,4,5,6,7,8}
             {\fill (2.5,\y+.5) circle (3pt);
               \draw (2.5,\y+.5) -- (2.5,\y+1.5);}
             \foreach \y in {4,5,6,7,8}
                      {\fill (6.25,\y+.5) circle (3pt);
                        \draw (6.25,\y+.5) -- (6.25,\y+1.5);}
                      \foreach \y in {1,2,3,4,5,6,7,8}
                               {\fill (6.75,\y+.5) circle (3pt);
                                 \draw (6.75,\y+.5) -- (6.75,\y+1.5);}
  \end{tikzpicture}
  \end{center}
  \caption{\label{fig:e2_page_MTO2} The $E_2 = E_\infty$ page for $\pi_\ast MT\Pin^{\tilde{c}+}$}
  \end{figure}

  There are no differentials. We can now read off the homotopy groups of $MT\text{Pin}^{\tilde{c}+}$
  \begin{thm}
    The low-dimensional homotopy groups of $MT\Pin^{\tilde{c}+}$ are given by 
  \begin{align}
    \pi_0 MT\text{Pin}^{\tilde{c}+} &= \Z/2\\
    \pi_1 MT\text{Pin}^{\tilde{c}+} &= 0 \\
    \pi_2 MT\text{Pin}^{\tilde{c}+} &= \Z \\
    \pi_3 MT\text{Pin}^{\tilde{c}+} &= \Z/2\Z\\
    \pi_4 MT\text{Pin}^{\tilde{c}+} &= (\Z/2\Z)^3
  \end{align}
  \end{thm}

\end{example}

\begin{example}
  The case $s = 3$. This corresponds to the homotopy groups of $MTG^+$ where $G^+$ is the stabilization \cite[Thm. 2.12]{freed_hopkins} of $G^+_n  = \Pin^+_n \times_{\{\pm 1\}} SU(2)$. We have the identification
  \begin{equation}
    MTG^+ \simeq M\Spin \sma S^{-3} \sma MO(3). 
  \end{equation}

    Again, by stability, we compute the $\mc{A}(1)$-module structure of $H^\ast (MO(3))$. Recall as a group it has generators
  \begin{equation}
  U, w_1 U, w^2_1 U, w_2 U, w_1 w_2 U, w^3_1 U, w_3 U, \cdots
  \end{equation}
  Starting with $U$, we begin applying elements of $\mc{A}(1)$. We obtain, as in the case of $H^\ast (MO(2))$
  \begin{align}
    \Sq^1(U) &= w_1 U\\
    \Sq^2 U &= w_2 U\\
    \Sq^2 \Sq^1 U &= (w_1^3  + w_1 w_2)U\\
    \Sq^2 \Sq^2 U & = (w^2_1 w_2 + w_1 w_2)U\\
  \end{align}
  It begins to differ from $H^\ast (MO(2))$ when we apply $\Sq^1 \Sq^2 U$ and obtain $\Sq^3 U = w_3 U$. Furthermore, $\Sq^2 (\Sq^3 U) = w^2_1 w_3 U$. Thus far we've generated the following $\mc{A}(1)$-module
  
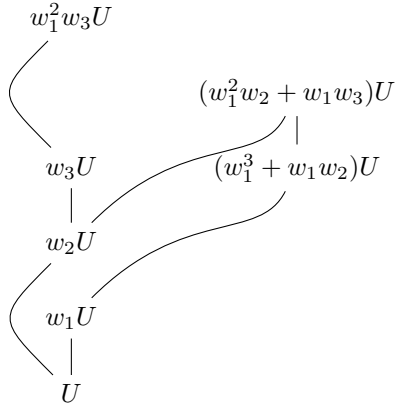
\begin{center}
    \begin{tikzpicture}
      \draw (0, 0) node(0){$U$};
      \draw (0,1) node(1){$w_1 U$};
      \draw (0,2) node(2){$w_2 U$};
      \draw (0, 3) node(3){$w_3 U$};
      \draw (0, 5) node(6){$w^2_1 w_3 U$};
      \draw (3, 3) node(4){$(w^3_1 + w_1w_2)U$};
      \draw (3, 4) node(5){$(w^2_1 w_2 + w_1 w_3)U$};
      \draw (0) -- (1);
      \draw (2) -- (3);
      \draw (0) .. controls (-1,1) .. (2);
      \draw (3) .. controls (-1,4) .. (6);
      \draw (1) .. controls (1.5,2.5) and (2.5,2) .. (4);
      \draw (2) .. controls (1.5,3.5) and (2.5,3) .. (5);
      \draw (4) -- (5);
    \end{tikzpicture}
  \end{center}

  The lowest degree element of $H^\ast (MO(3))$ that has not appeared is $w^2_1 U$. Starting with that element, we generate a full $\mc{A}(1)$-module (this computation is omitted). After that, the lowest degree element that has not appeared is $w_1 w_3 U$. Now we note that
  \begin{equation}
  \Sq^1 (w_1 w_3 U) = w^2_1 w_3 U
  \end{equation}
  which means our picture above becomes
    
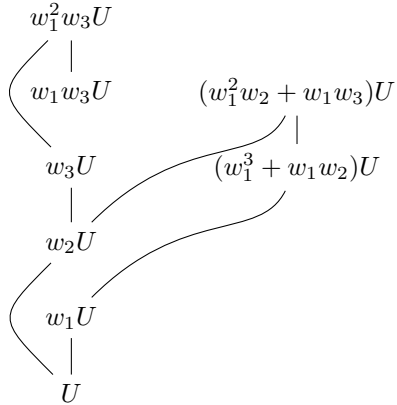
\begin{center}
    \begin{tikzpicture}
      \draw (0, 0) node(0){$U$};
      \draw (0,1) node(1){$w_1 U$};
      \draw (0,2) node(2){$w_2 U$};
      \draw (0, 3) node(3){$w_3 U$};
      \draw (0, 5) node(6){$w^2_1 w_3 U$};
      \draw (3, 3) node(4){$(w^3_1 + w_1w_2)U$};
      \draw (3, 4) node(5){$(w^2_1 w_2 + w_1 w_3)U$};
      \draw (0,4) node(7) {$w_1 w_3 U$};
      \draw (0) -- (1);
      \draw (2) -- (3);
      \draw (0) .. controls (-1,1) .. (2);
      \draw (3) .. controls (-1,4) .. (6);
      \draw (1) .. controls (1.5,2.5) and (2.5,2) .. (4);
      \draw (2) .. controls (1.5,3.5) and (2.5,3) .. (5);
      \draw (4) -- (5);
      \draw (7) -- (6);
    \end{tikzpicture}
      \end{center}
    Furthermore, $\Sq^2 (w_1 w_3 U)$ is nonzero, so the diagram continues upward. There is one final basis element left in the $\Z/2$ vector space generated by $w^4_1, w^2_1 w_2, w_1 w_3, w^2_2$, this element generates another $\mc{A}(1)$. We thus have the picture Fig. \ref{fig:MO3} for the low-degree $\mc{A}(1)$-module structure for $H^\ast (MO(3))$.

\begin{figure}
  \begin{center}
    \begin{tikzpicture}[scale=.5]
      \fill (2,3) circle (3pt);
      \fill (2,4) circle (3pt);
      \foreach \y in {0,1,2,3,4,5,6,7,8,9}
               {\fill (0,\y) circle (3pt);}
               \sqtwoL(0,0,black);
               \sqtwoL(0,3,black);
               \sqtwoR(0,4,black);
               \sqtwoR(0,7,black);
               \sq1 (0,0,black);
               \sq1 (0,2,black);
               \sq1 (0,4,black);
               \sq1 (0,6,black);
               \sq1 (0,8,black);
               \sqtwoCR (0,2,black);
               \sqtwoCR (0,1,black);
               \sq1 (2,3,black);
               \A1 (4,2);
               \A1 (8,4);
    \end{tikzpicture}
  \end{center}
  \caption{\label{fig:MO3} The $\mc{A}(1)$-module structure of $S^{-3} \sma MO(3)$}
\end{figure}

We now use the Adams spectral sequence to compute the homotopy groups of $S^{-3} \sma MO(3)$. We compute the resolution for the $\mc{A}(1)$-module. We focus on the leftmost module in the picture, which can be viewed as an extension of $\mc{A}(1)$-modules:
\begin{figure}
  \begin{center}
    \begin{tikzpicture}[scale=.5]
      \fill (0,1) circle (3pt);
      \fill (0,3) circle (3pt);
      \fill (0,4) circle (3pt);
      \sq1 (0,3,black);
      \sqtwoR (0,1,black);
       \fill (6,3) circle (3pt);
       \fill (6,4) circle (3pt);
      \foreach \y in {0,1,2,3,4,5,6,7,8,9}
               {\fill (4,\y) circle (3pt);}
               \sqtwoL(4,0,black);
               \sqtwoL(4,3,black);
               \sqtwoR(4,4,black);
               \sqtwoR(4,7,black);
               \sq1 (4,0,black);
               \sq1 (4,2,black);
               \sq1 (4,4,black);
               \sq1 (4,6,black);
               \sq1 (4,8,black);
               \sqtwoCR (4,2,black);
               \sqtwoCR (4,1,black);
               \sq1 (6,3,black);
                     \foreach \y in {0, 2, 3, 4, 5, 6, 7, 8,9}
               {\fill (8,\y) circle (3pt);}
               \sqtwoL(8,0,black);
               \sqtwoL(8,3,black);
               \sqtwoR(8,4,black);
               \sqtwoR(8,7,black);
               \sq1 (8,2,black);
               \sq1 (8,4,black);
               \sq1 (8,6,black);
               \sq1 (8,8,black);
               \draw[->,green] (0,1) -- (4,1);
               \draw[->,green] (0,3) -- (6,3);
               \draw[->,green] (0,4) -- (6,4);
    \end{tikzpicture}
  \end{center}
  \caption{\label{fig:extension} The extension $\Sigma^1 Q \to H^\ast (S^{-3} \sma MO(3)) \to M$}
  \end{figure}
    Let us write this extension as a short exact sequence
    \begin{equation}
    \Sigma^1 Q \to H^\ast (S^{-3} \sma MO(3)) \to M
    \end{equation}
    which induces a long exact sequence
    \begin{equation}
    \xymatrix{
      \ext^{s,\ast}(M, \Z/2) \ar[r] & \ext^{s,\ast} (S^{-3} \sma MO(3), \Z/2)\ar[r] & \ext^{s,\ast} (\Sigma^1 J, \Z/2) \ar[dll]^\delta \\
      \ext^{s+1,\ast}(M, \Z/2) \ar[r] & 
    }
    \end{equation}

    This sequence gives that $\delta$ is an isomorphism when it is non-zero, and so $\ext^{s,t} (H^\ast (S^{-3} \sma MO(3)), \Z/2)$ is what is left after quotienting. We represent this in the diagram Fig. \ref{fig:MTG+}. 
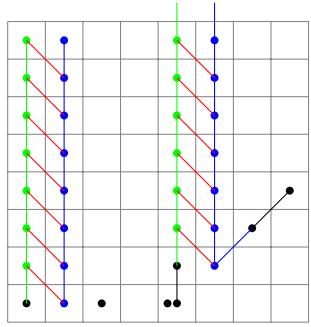
\begin{figure}
    \begin{center}
  \begin{tikzpicture}[scale=.5]
    \draw[step=1cm,gray,very thin] (0,0) grid (8,8);
    \fill (2.5,0.5) circle (3pt);
    \fill (0.5,0.5) circle (3pt);
    \fill (4.25,0.5) circle (3pt);
    \fill (4.5,0.5) circle (3pt);
    \fill (4.5,1.5) circle (3pt);
    \draw (4.5,0.5) -- (4.5,1.5);
    \fill (6.5,2.5) circle (3pt);
    \fill (7.5,3.5) circle (3pt);
    \draw (6.5,2.5) -- (7.5,3.5);
    \foreach \y in {1,2,3,4,5,6,7}
             {\fill[green] (0.5,\y+.5) circle (3pt);}
    \foreach \y in {1,2,3,4,5,6,7}
             {\draw[green] (0.5,\y-.5)--(.5,\y+.5);}
     \foreach \y in {0,1,2,3,4,5,6,7}
             {\fill[blue] (1.5,\y+.5) circle (3pt);}
     \foreach \y in {1,2,3,4,5,6,7}
              {\draw[blue] (1.5,\y-.5)--(1.5,\y+.5);}
     \foreach \y in {1, 2, 3, 4, 5, 6,7}
              {\draw[green] (4.5,\y+.5)--(4.5,\y+1.5);}
     \foreach \y in {2,3,4,5,6,7}
              {\fill[green] (4.5,\y+.5) circle (3pt);}
     \foreach \y in {1,2,3,4,5,6,7}
              {\fill[blue] (5.5,\y+.5) circle (3pt);}
     \foreach \y in {1,2,3,4,5,6,7}
              {\draw[blue] (5.5,\y+.5) -- (5.5,\y+1.5);}
     \draw[blue] (5.5,1.5) -- (6.5,2.5);
     \foreach \y in {0,1,2,3,4,5,6}
              {\draw[red] (1.5,\y+.5) -- (0.5,\y+1.5);}
     \foreach \y in {1,2,3,4,5,6}
              {\draw[red] (5.5,\y+.5) -- (4.5,\y+1.5);}
  \end{tikzpicture}
    \end{center}
    \caption{\label{fig:MTG+} The Adams spectral sequence for $\pi_\ast MTG^+$}
    \end{figure}
where green is $\ext^{\ast,\ast} (M, \Z/2)$, blue is $\ext^{\ast,\ast}(\Sigma^1 Q, \Z/2)$, the red line represents $\delta$ (recall the reindexing) and black is the quotient, i.e. $\ext^{\ast,\ast} (H^\ast (S^{-3} \sma MO(3)), \Z/2)$.

There are no differentials, and we read off the homotopy groups as usual.

\begin{thm}
  The low-dimensional homotopy groups of $MTG^+$ are
  \begin{align}
    \pi_0 MTG^+ &= 0\\
    \pi_1 MTG^+ &= 0\\
    \pi_2 MTG^+ &= \Z/2 \\
    \pi_3 MTG^+ &= 0\\
    \pi_4 MTG^+ &= \Z/2 \oplus \Z/4 
  \end{align}
\end{thm}

\end{example}

\section{New Computations}

The Freed-Hopkins classification allows us to revisit computations already in the physics literature and compute many more invariants in the case of external symmetry groups.

We first consider the case of bosonic field theories with external symmetry group $G$. In this case field theories with $n$ spatial dimensions and symmetry group $G$ are classified by
\begin{equation}
[MTH \sma BG_+, \Sigma^{n+1} I_{\Z}]
\end{equation} 
where $H$ is one of $O$, or $SO$, depending on whether or not time reversal is present \cite{freed_hopkins}. These computations are easily dispatched by using classifical results of Thom \cite{thom}, Wall \cite{wall} and the Atiyah-Hirzebruch spectral sequence. 

We move on to consider various fermionic computations with a symmetry group $G$. These are classified by
\begin{equation}
[MTH \sma BG_+, \Sigma^{n+1} I_\Z]
\end{equation}
where $H$ is typically $\Spin$. These computations are mostly carried out using $ko$-homology theory.

\subsection{Bosonic, with Time Reversal}

Bosonic theories with time reversal in $n = d+1$ space time dimensions are classified by 
\begin{equation}
[MO, \Sigma^{n+1} I_{\Z}] \simeq [MO, \Sigma^{n+1} I_{\Z} \langle 0, \dots, \infty\rangle] \simeq \pi_0 \Sigma^{n+1} I_{\Z} MO \langle 0, \dots, \infty \rangle
\end{equation}
The unoriented cobordism ring was completely computed by Thom \cite{thom}
\begin{equation}
  \Omega^O_\ast \cong \Z/2[x_n | n \neq 2^i-1]
\end{equation}
and furthermore, there is a spectrum-level splitting of $MO$ into Eilenberg-Maclane spaces
\begin{equation}
MO \simeq \bigvee \Sigma^{k_i} H\mbf{F}_2. 
\end{equation}
The number of copies and suspension degree of each $H\mbf{F}_2$ is determined by the algebra structure of $\Omega^O_{\ast}$. For example there are 2 elements of degree 4: $x^2_2, x_4$ and so 3 copies of $\Sigma^4 H \mbf{F}_2$. 

We have $I_{\Z} H\mbf{F}_2 \simeq \Sigma^{-1} H\mbf{F}_2$ (this follows directly from a computation of homotopy groups) and thus the Anderson dual of $MO$ is
\begin{equation}
I_{\Z} MO = \prod \Sigma^{-k_i - 1} H \mbf{F}_2. 
\end{equation}
The first few homotopy groups of this are
\begin{equation}
\pi_{-1} = \Z/2 \ \ \ \pi_{-2} = 0 \ \ \ \pi_{-3} = \Z/2  \ \ \ \pi_{-4} = 0 \ \ \pi_{-5} = (\Z/2)^2
\end{equation}

\begin{example}
For $d = 0$ (i.e. $d$ spatial dimensions), $\Sigma^2 I_{\Z} MO \langle 0, \dots, \infty \rangle \simeq \Sigma^1 H \mbf{F}_2$. The components of this are trivial.

For $d = 1$
\begin{equation}
\Sigma^3 I_{\Z} MO \langle 0 \dots \infty \rangle \simeq \Sigma^2 H \mbf{F}_2 \vee H\mbf{F}_2 
\end{equation}
and so $\pi_0$ is $\Z/2$.

For $d = 2$
\begin{equation}
\Sigma^4 I_{\Z} MO \langle 0 \dots \infty \rangle \simeq \Sigma^3 H\mbf{F}_2 \vee \Sigma^1 H \mbf{F}_2 
\end{equation}
and $\pi_0$ is trivial.

For $d = 3$
\begin{equation}
\Sigma^5 I_{\Z} MO \langle 0 \dots \infty \rangle \simeq \Sigma^4 H\mbf{F}_2 \vee \Sigma^2 H \mbf{F}_2 \vee H\mbf{F}_2 \vee H\mbf{F}_2
\end{equation} 
where we have two copies of $H \mbf{F}_2$ because there are two generators in degree 4 of $\Omega^0_\ast$ : $x^2 x^2$ and $x^4$.

For $d = 4$
\begin{equation}
\Sigma^6 I_{\Z} MO \langle 0, \dots, \infty \rangle \simeq \Sigma^5 H\mbf{F}_2 \vee \Sigma^3 H\mbf{F}_2 \vee \Sigma^1 (H \mbf{F}_2 \vee H \mbf{F}_2) \vee H \mbf{F}_2.
\end{equation}
\end{example}

The above example gives the following classification
\begin{align}
  d = 0 &\qquad 0 \\
  d = 1 &\qquad \Z/2\\
  d = 2 &\qquad 0\\
  d = 3 &\qquad (\Z/2)^2 \\
  d = 4 &\qquad \Z/2 
\end{align}

This is in agreement with \cite{kapustin_bosonic} (and see references therein).

\subsection{Bosonic, Time Reversal, $U(1)$-symmetry}

This phase in $n= d+1$ spatial dimensions will be classified by
\begin{equation}
[MO \sma BU(1)_+, \Sigma^{d+2} I_{\Z}] \cong [\C P^\infty_+, \Sigma^{d+2} I_{\Z} MO \langle 0, \dots, \infty \rangle]. 
\end{equation}
Since we have computed $\Sigma^{d+2} I_{\Z} MO$ above, and found it to be a variety of shifts of $H\mbf{F}_2$, the above amounts to computing $H^\ast (\C P^\infty; \Z/2)$ in various degrees. For example, for $d = 3$, we compute
\begin{equation}
[\C P^\infty_+, \Sigma^4 H \mbf{F}_2 \vee \Sigma^2 H \mbf{F}_2 \vee \mbf{F}_2 \vee \mbf{F}_2] \cong (\Z/2)^4
\end{equation}
since $H^4(\C P^\infty)$, $H^2 (\C P^\infty)$ and $H^0(\C P^\infty)$ are each copies of $\Z/2$. In low dimenions, we have the following classification
\begin{align}
  d &= 0 \qquad 0\\
  d &= 1 \qquad (\Z/2)^2 \\
  d &= 2  \qquad 0\\
  d &= 3 \qquad (\Z/2)^4\\
  d &= 4 \qquad \Z/2 \\
\end{align}

\begin{rmk}
It is clear that for any group $G$, $[MO \sma BG_+, \Sigma^{d+2} I_{\Z}]$ is a fairly simple computation in general. 
\end{rmk}

\subsection{Bosonic, No Time-Reversal, $U(1)$-symmetry}

The situation of an $n$-dimensional bosonic SPT lacking time reversal symmetry but with a $U(1)$-symmetry amounts to computation of the homotopy groups
\[
[MSO \sma \C P^\infty_+, \Sigma^{n+1} I_{\Z}].
\]
This phase is discussed in the physics literature in \cite{wang_gu_wen}. 

The classification of this phase in low dimensions is easily accomplished with the Atiyah-Hirzebruch spectral sequence. Recall that the homotopy groups of $MSO$ are given by \cite{wall}:
\begin{align}
  \pi_0 MSO &= \Z \\
  \pi_1 MSO &= 0 \\
  \pi_2 MSO &= 0\\
  \pi_3 MSO &= 0\\
  \pi_4 MSO &= \Z\\
  \pi_5 MSO &= \Z/2\\
  \pi_6 MSO &= 0
\end{align}

Using Eqn. \ref{anderson_exact_sequence} it is easy to compute the homotopy groups of $\Sigma^{n+1} I_{\Z} MSO$. We note that $\Sigma^2 I_{\Z} MSO \simeq \Sigma^2 H\Z$ and $\Sigma^3 I_Z MSO \simeq \Sigma^3 H\Z$,  $\Sigma^4 I_Z MSO$ is such that $\pi_0 \Sigma^4 I_{\Z} MSO \cong \Z$ and $\pi_4 I_{\Z} MSO \cong \Z$ with vanishing homotopy groups elsewhere. The spectrum $\Sigma^5 I_{\Z} MSO$ has homotopy groups $\Z$ in degree 1 and 5 with vanishing homotopy groups elsewhere. Finally, $\Sigma^6 I_{\Z} MSO$ has $\pi_0 = \Z/2$, $\pi_2 = \Z$ and $\pi_6 = \Z$ with all other homotopy groups vanishing. Using the Atiyah-Hirzebruch spectral sequence and noting that no differentials can occur in these cases we get the classification in $d$ spatial dimensions
\begin{align*}
  d&=0 \qquad \Z \\
  d&=1 \qquad 0 \\
  d&=2 \qquad \Z\oplus \Z \\
  d&=3 \qquad 0 \\
  d&=4 \qquad \Z \oplus \Z \oplus \Z/2
\end{align*}
This is in agree with \cite{wang_gu_wen}. 

\subsection{Fermionic Systems}

According to the Freed-Hopkins classification, fermionic systems are classified by $[MTH \sma BG_+, \Sigma^{n+1}I_\Z]$ where $H$ is some extension of $O$ or $SO$ by $\Z/2, U(1)$ or $SU(2)$ whose identity component is related to $\Pin$ or $\Spin$ (the specifics will not be important so us so we are bit vague). The group extensions correspond to various symmetries (time reversal symmetry, particle hole duality) that we will not consider below, so $H = \Spin$. In many of the cases we want to analyze, rather than computing the full homotopy groups, we will be interested in the classification in a particular dimension, both for ease of computation and comparison to the physics literature. Before starting the computations it is worthwhile to recall a few facts.

First, we remarked above that homologically, $M\Spin$ resembles $ko$. However, a stronger statement is true: the Atiyah-Bott-Shapiro map $M\Spin \to ko$ is an isomorphism in degrees $\leq 7$. Thus, in the physically relevant dimensions, $ko$ is a perfectly reasonable stand-in for $M\Spin$, reducing the computations to
\begin{equation}
[ko \sma BG_+, \Sigma^{n+1} I_{\Z}].
\end{equation}
For many groups, $\pi_\ast (ko \sma BG_+)$, the $ko$-homology of $BG$, has been computed (see \cite{bruner_greenlees}). There are a nevertheless a number of interesting computations not covered in that book, and some computations that, in very specific cases, can be done with less technology. Above, we used the Adams spectral sequence for computations. In the case with an external symmetry group, some of those computations will prove difficult. For example, in the case $G = G_1 \times G_2$, the $\mc{A}(1)$-module structure of $H^\ast (BG)$ becomes unwieldy. Instead, we resort to tricks.

The group $[ko \sma BG_+, \Sigma^{n+1} I_{\Z}]$ is equivalent to $[BG_+, \Sigma^{n+1} I_{\Z} ko]$, so it falls on us to figure out the Anderson dual $I_{\Z}ko$. The periodic real $K$-theory spectrum $KO$ is in fact Anderson self-dual: $I_{\Z} KO \simeq \Sigma^4 KO$ \cite{freed_moore_segal, heard_stojanoska}.  We then consider the cofiber sequence
\begin{equation}
ko \to KO \to C
\end{equation}
where $C$ is the cofiber and a coconnective spectrum. Applying $I_{\Z}$ gives a (co)fiber sequence of spectra
\begin{equation}
I_{\Z} C \to \Sigma^4 KO \to I_{\Z} ko 
\end{equation}
from which it is easy to see that
\begin{equation}
I_{\Z} ko \simeq (\Sigma^4 KO)\langle -\infty, \dots, 0\rangle. 
\end{equation}
We will be interested in physical dimensions $(2+1)$ and $(3+1)$. In dimesion $(2+1)$, we have
\begin{equation}\label{ko4}
(\Sigma^4 I_{\Z} ko)\langle 0, \dots, \infty \rangle \simeq ko \langle 0, \dots, 4 \rangle
\end{equation}
so that in this dimension we can use computations of $ko$-cohomology and the Atiyah-Hirzebruch to aide us.

In order to use the Atiyah-Hirzebruch spectral sequence in some of the examples below, we recall Adams celebrated result on $KO(\R P^n)$.

\begin{thm}[Adams, \cite{adams_vector_fields}] 
  The reduced $KO$ cohomology of $n$-dimensional real projective space is given by
  \begin{equation}\label{koRPn}
  \widetilde{KO}(\R P^n) \cong \Z/2^f
  \end{equation}
  where $f$ is the count of numbers that are 0, 1, 2 or 4 mod 8 between $1$ and $f$. 
\end{thm}

This theorem essentially says that whatever the extension problems are in the Atiyah-Hirzebruch spectral sequence, we resolve them with maximal torsion.

\subsection{Fermionic, Symmetry Group $C_2$, dimension 5}

Classification of fermionic field theories with external symmetry group $G$ amounts to a computations of
\begin{equation}
[MSpin \sma (BC_2)_+, \Sigma^{n+1}I_{\mbf{Z}}]. 
\end{equation}
There are a number of simplifications that will help. First, since $BC_2 \simeq \R P^\infty$ is 2-local, we might as well just use 2-local $MSpin$, which is known to split into $ko$ and shifts of $ko$ above degree higher than 8 \cite{anderson_brown_peterson}. That is,
\begin{equation}
MSpin\langle 0, \dots, 7 \rangle \simeq ko \langle 0, \dots, 7 \rangle. 
\end{equation}
Furthermore, $KO$ is Anderson self dual in the sense that $I_{\mbf{Z}} KO \simeq \Sigma^4 KO$ \cite{heard_stojanoska, freed_moore_segal}. Altogether, this gives us that
\begin{equation}
[MSpin \sma \R P^\infty_+, \Sigma^5 I_{\Z}] \cong [\R P^\infty_+, \Sigma^1 ko \langle 0 \dots 4 \rangle] 
\end{equation}

This computation is the computation, of $k\langle 0, \dots, 4 \rangle^1 (\R P^\infty)$, which is amenable to an attack by the Atiyah-Hirzebruch spectral sequence and the known computations of $ko(\R P^n)$ \cite{adams_vector_fields}.

The $E^2$-page that computes $ko\langle 0, \dots, 4 \rangle$-cohomology of $\R P^\infty$ is
\begin{equation}
\xymatrix@C=.3cm@R=.3cm{
 \ast &0  & 1 & 2   &  3 & 4    & 5 & 6\\
0  & \Z & 0 & \Z/2 & 0 & \Z/2 & 0 & \Z/2\\
-1 & \Z/2 & \Z/2\ar[drr] & \boxed{\Z/2}\ar[drr] & \Z/2 & \Z/2 & \Z/2 & \Z/ 2\\
-2 & \Z/2 & \Z/2 & \Z/2 & \boxed{\Z/2}\ar[drr]\ar[ddrrr] & \Z/2 & \Z/2 & \Z/ 2\\
-3 & 0    &   0  & 0    & 0    &   0 &    0  &   0 \\
-4 &  \Z & 0 & \Z/2 & 0 & \Z/2 & 0 & \Z/2\\ 
}
\end{equation}

Note that since $ko$ has no negative homotopy groups, we only have negative entries in the $q$-coordinate. Note also that the 0th and 4th rows are $H^\ast (BC_2; \Z)$. The group that we want to compute lies along the diagonal with the indicated boxed elements. The $d_2 : E^{1,-1}_2 \to E^{3,-2}_2$ is given by a stable cohomology operation $H^1(\R P^\infty;\Z/2) \to H^3(\R P^\infty;\Z/2)$ --- and there is no such cohomology operations. Thus that particular $d_2$ is 0. The $d_2: E^{2,-1}_2 \to E^{4,-2}_2$ is given by $\Sq^2$, and so is an isomorphism. Thus, of the boxed $\Z/2$s only the one at $(3,-2)$ survives to the $E_3$ page.

We now consider the indicated $d_3$ from $E^{3,-2}_3 \to E^{6,-4}_3$. From the Postnikov tower for $ko$ (see e.g. the appendices to \cite{bruner_greenlees}, this is
\begin{equation}
\tilde{\beta} \circ \Sq^2 : H^3 (\R P^\infty; \Z/2) \to H^6 (\R P^\infty; \Z)
\end{equation}
$H^3 (\R P^\infty; \Z/2)$ is of course generated by $x^3$ with $x$ the distinct non-zero element in $H^1 (\R P^\infty; \Z/2)$. We know that $\Sq^2 (x^3) = x^5$, so $\Sq^2$ is an isomorphism in this case. Furthermore, the integral Bockstein is an isomorphism when the degree is odd. Thus $d_3$ is an isomorphism. We have the result that $[MSpin \sma \R P^\infty_+, \Sigma^5 I_{\Z}]$ vanishes.

\begin{rmk}
  A very similar argument shows that for symmetry group $\Z/n$ we get nothing interesting --- there are no non-trivial phases. For $n$ odd this is clear, since all the relevant cohomology groups vanish --- we get nothing along the appropriate diagonal. In the case when $n$ isn't odd, write $\Z/n$ as a product of 2-torsion and odd torsion. Then an argument nearly the same as above gives that $[MSpin \sma BC_{n+}, \Sigma^5 I_{\Z}] \cong 0$.
\end{rmk}

\begin{rmk}
One could hope for something interesting in the case of $C_2 \times C_2$. There are no interesting theories with this group either. 
\end{rmk}

\subsection{Symmetry Group $C_2 \times C_4$}

We showed above that in $(3+1)$ dimensions, theories with external symmetry group $C_n$ and $C_2 \times C_2$ vanish. The next most complicated group is $C_2 \times C_4$. Here we show there is a non-trivial phase which agrees with work of Meng Cheng \cite{cheng}. 

Here we essentially follow the same procedure as above, except we must compute $H^\ast (BC_2 \times BC_4; \Z)$. First, note that $H^\ast (BC_2; \Z/2) \cong H^\ast (BC_4; \Z/2)$ with the isomorphism being given by map induced by the inclusion $i: BC_2 \to BC_4$. Let the generator for $H^\ast (BC_2; \Z/2)$ be $x$ and for $H^\ast (BC_4; \Z/2)$ be $\alpha$.  The relevant integral cohomology group is $H^6(BC_2 \times BC_4; \Z) \cong (\Z/2)^3 \oplus \Z/4$ by the K\"{u}nneth formula. The K\"{u}nneth formula of course also gives $H^\ast (BC_2 \times BC_4; \Z/2)$. The relevant portion of the spectral sequence is

\begin{align*}
\xymatrix@C=.5cm@R=.5cm{
  \ast & 0 & 1 & 2 & 3 & 4 & 5 & 6\\
  0 & & & &\\
  -1& & &(\Z/2)^3 \ar[drr]&\\
  -2& & & & (\Z/2)^4\ar[drr]\ar[ddrrr] & (\Z/2)^5\\
  -3& & & & & & & \\
  -4& & & & & & & (\Z/2)^3 \oplus \Z/4\\ 
}
\end{align*}
The indicated $d_2: H^2(BC_2 \times BC_4; \Z/2) \to H^4 (BC_2 \times BC_4; \Z/2)$ is injective, so that entry vanishes on the $E_3$-page. 

The $d_3 = \widetilde{\beta} \circ \Sq^2 : H^3(BC_2 \times BC_4; \Z/2) \to H^6(BC_2 \times BC_4; \Z)$ certainly can't be surjective. So, something survives to the $E_\infty$-page. Thus, there is a non-trivial $BC_2 \times BC_4$-phase.

\subsection{Fermionic System, Symmetry Group $(C_2)^{\times k}$, dimension (2+1)}

In this section we consider fermionic systems in dimension $(2+1)$ with symmetry group $(C_2)^{\times k}$. The classification of such phases is given by the group
\begin{equation}
[M\Spin \sma (BC_2)^{\times k}, \Sigma^4 I_{\Z}]
\end{equation}
In order to employ the Atiyah-Hirzebruch spectral sequence in this case we have to know the integral cohomology fo $(\R P^\infty)^{\times k}$. We further have to know a few differentials in the Atiyah-Hirzebruch spectral sequence.

Let us consider the case $k = 1$. In this case, the classification is given by $[\R P^\infty_+ \sma MSpin, \Sigma^4 I_{\Z}]$, which amounts to $[\R P^\infty_+, ko \langle 0, \dots, 4 \rangle]$ by Eqn. \ref{ko4} Since $ko\langle 0, \dots, 4\rangle$ has no cells above degree 4, we note that this is the same as computing $[\R P^4_+, ko \langle 0, \dots,4 \rangle]$, i.e.  $ko \langle 0, \dots, 4 \rangle (\R P^4)$. However, the Postnikov tower is equipped with a map $ko \to ko \langle 0, \dots, 4, \rangle$ and pulling back by this we get a map $ko \langle 0, \dots, 4 \rangle (\R P^4) \to ko (\R P^4)$, which is again an isomorphism. Of course, since $\R P^4$ is space, we also have $ko(\R P^4) \cong KO(\R P^4)$. By Eqn. \ref{koRPn} this is $\Z/8$. For future reference, we also note that this group is generated by the canonical (virtual) bundle $[L]-1$ where $L$ is the line bundle classified by the inclusion $\R P^n \to \R P^\infty \simeq K(\Z/2, 1)$ \cite{adams_vector_fields}. 

Now, we consider the case $k = 2$. We are reduced to computing $ko\langle 0, \dots, 4 \rangle (\R P^\infty \times \R P^\infty)$. Our preferred tool will be the Atiyah-Hirzebruch spectral sequence, to see the 2-torsion in the group, and then known computations of $\widetilde{KO}(\R P^n)$ to solve extension problems. The $E_2$-page for the spectral sequence is

\begin{equation}
\xymatrix@C=.5cm@R=.5cm{
  \ast & 0 & 1 & 2 & 3 & 4 & 5\\
  0 & & & &\\
  -1&(\Z/2) & \boxed{(\Z/2)^2} & (\Z/2)^3 & (\Z/2)^4 & (\Z/2)^5\\
  -2&(\Z/2) & (\Z/2)^2 & \boxed{(\Z/2)^3} & (\Z/2)^4 & (\Z/2)^5\\
  -3& & & & & & & \\
  -4& & & & &\boxed{(\Z/2)^3} & & \\ 
}
\end{equation}
We note that all of the pertinent differentials must vanish. In particular, $d_3 = \widetilde{\beta}\circ \Sq^2: H^2((\R P^\infty)^2; \Z/2) \to H^5((\R P^\infty)^2; \Z)$ vanishes, because $\widetilde{\beta}$ vanishes on even classes. We now have to resolve extension problems. We do this by finding explicit generators for $ko\langle 0, \dots, 4 \rangle ((\R P^\infty)^2)$ and showing that these account for all the elements the spectral sequence is seeing.

As in the case of $\R P^\infty$ we note that we are in fact dealing with just the $5$-skeleton of $(\R P^\infty)^2$, so that we are effectively computing in $\widetilde{KO}((\R P^4 \times \R P^4)^{(5)})$. To deal with this group, we consider bundles over $\R P^4 \times \R P^4$ and check that the bundles vanish on the appropriate parts of the skeleton. We note that the total Stiefel-Whitney class descends to a natural map $\mbf{w}: KO(X) \to H^\ast(X;\Z/2)$. 

We first note that we can pullback line bundles under the projections $\pi_1, \pi_2: \R P^4 \times \R P^4 \to \R P^4$ to obtain $\pi^\ast_1 L$ and $\pi^\ast_2 L$ in $\widetilde{KO}(\R P^4 \times \R P^4)$.   Both $\pi^\ast_1 L $ and $\pi^\ast_2 L $ will satisfy the same relations as $L$ in $KO(\R P^4)$, and thus generate groups isomorphic to $\Z/8$. The remainder of $\widetilde{KO}((\R P^4 \times \R P^4)^{(5)})$ is thus either $\Z/2 \oplus \Z/2$ or $\Z/4$ by a count of 2-torsion. We show that its the latter.

 For brevity let $R = ko \langle 0, \dots, 4 \rangle$ (this notation is consistent with \cite{freed_anomalies}). We look at the element $\lambda_{1,2} = ([\pi^\ast_1 L] - 1)([\pi^\ast_2 L]-1)$. Then we consider the diagram
\begin{equation}
\xymatrix{
  R(\R P^4) \ar[d]_{\mbf{w}}\ar[r]^{\pi^\ast} & R(\R P^4 \times \R P^4) \ar[d]^{\mbf{w}}\\
  H^\ast(\R P^4; \Z/2) \ar[r]^{\pi^\ast} & H^\ast (\R P^4 \times \R P^4; \Z/2)
}
\end{equation}
and the image of $\lambda_{1,2}$ under $\mbf{w}$. Let $\mbf{w}([L_1]) = a_1$ and $\mbf{w}([L_2]) = a_2$ where $a_1, a_2$ are the generators of $H^1(\R P^4 \times \R P^4; \Z/2)$ coming from the image of $\pi^\ast$. Then
\begin{align*}
  \mbf{w}(\lambda_{12}) &= \mbf{w}([\pi^\ast_1 L \otimes \pi^\ast_2 L_2] - \pi^\ast_1 L - \pi^\ast_2 L + 1)\\
  &=\mbf{w}([\pi^\ast_1 L \otimes \pi^\ast_2 L_2])\overline{\mbf{w}}(\pi^\ast_1 L) \overline{\mbf{w}}(\pi^\ast_2 L)\\
  &= (1 + a_1 + a_2)(1+a_1 + \text{h.o.t})(1+a_2 + \text{h.o.t})\\
  &= 1\  \text{mod terms deg 2 or higher}
\end{align*}
So, for example, we see that $\lambda_{1,2}$ vanishes on the 1-skeleton.

Now, $\lambda^2_{1,2} = -2\lambda_{1,2}$ since each of $([\pi^\ast_1 L]-1)$ and $([\pi^\ast_2 L]-1)$ satisfy this identity. We then check that $\mbf{w}(2 \lambda_{1,2})$, $\mbf{w}( 3 \lambda_1,2)$ are non-zero, but $\mbf{w}(4 \lambda_{1,2}) = 0$. To check each of these, we compute
\begin{equation}
\mbf{w}(k \lambda_{12}) = \mbf{w}([\pi^\ast_1 L \otimes \pi^\ast_2 L_2])^k\overline{\mbf{w}}(\pi^\ast_1 L)^k \overline{\mbf{w}}(\pi^\ast_2 L)^k
\end{equation}
The computations are tedious by hand, but we find $\mbf{w}(4\lambda_{1,2})$ only involves terms of degree higher than 4 in $H^\ast(\R P^4 \times \R P^4;\Z)$ and so the bundle $4 \lambda_{1,2}$ vanishes on the 4 skeleton. Thus, $\lambda_{1,2}$ generates a $\Z/4$ and completes the proof of the following lemma.

\begin{lem}
$ko \langle 0, \dots, 4 \rangle (\R P^\infty \times \R P^\infty) \cong (\Z/8)^2 \oplus \Z/4$
\end{lem}

The case for $(\R P^\infty)^{\times k}$ for $k \geq 3$ is very similar. We can easily compute that
\begin{align}
H^1((\R P^\infty)^{\times k}; \Z/2) &\cong (\Z/2)^{k} \\ H^2((\R P^\infty)^{\times k}; \Z/2) &\cong (\Z/2)^{\binom{k}{2} + \binom{k}{1}} \\ H^4((\R P^\infty)^{\times k}; \Z) &\cong (\Z/2)^{\binom{k}{3} + \binom{k}{2} + \binom{k}{1}}
\end{align}
where the last is computed by induction and the K\"{u}nneth formula.

Resolving the extensions proceeds as before. We consider projections $\pi_i : (\R P^\infty)^{\times k} \to \R P^\infty$. We get $\binom{k}{1}$ copies of $\Z/8$ generated by $([\pi^\ast_i L]-1)$, $\binom{k}{2}$ copies of $\Z/4$ generated by $([\pi^\ast_i L]-1)([\pi^\ast_j L] - 1)$ where $i \neq j$. The computations that these generate what we claim are exactly analogous to the above. Finally, we get $\binom{k}{3}$ copies of $\Z/2$ generated by $\lambda_{ijk}:=([\pi^\ast_i L]-1)([\pi^\ast_j L]-1)([\pi^\ast_k L] -1)$ with $i, j, k$ pairwise distinct. We need to show that 2 times this bundle vanishes. It suffices to compute $\mbf{w}(\lambda_{ijk})$:
\begin{align}
  \mbf{w}(2 \lambda_{ijk}) &= \mbf{w}(2[\pi^\ast_i L]-1)([\pi^\ast_j L]-1)([\pi^\ast_k]-1)\\
  &\mbf{w}(2(\pi^\ast_i L \otimes \pi^\ast_j L \otimes \pi^\ast_k L - \pi^\ast_i L \otimes \pi^\ast_j L - \pi^\ast_j L \pi^\ast_k L - \pi^\ast_i L \pi^\ast_k L + \pi^\ast_i L + \pi^\ast_j L + \pi^\ast_k L -1))\\
  &= (1+a_i + a_j + a_k)^2(1+a_i+a_j+\cdots)^2(1+a_i + a_k+\cdots)^2\\
  &(1+a_j+a_k+\cdots)^2(1+a_i)^2(1+a_j)^2(1+a_k)^2\\
  &= 1 \ \text{mod 2 and mod terms of deg 5 or higher}
\end{align}
where the final computation is a computer algebra computation. We thus have

\begin{prop}
  We have
  \begin{equation}
  ko\langle 0, \dots, 4 \rangle ((\R P^\infty)^{\times k}) \cong (\Z/8)^{k} \oplus (\Z/4)^{\binom{k}{2}} \oplus (\Z/2)^{\binom{k}{3}}
  \end{equation}
\end{prop}

This result agrees with results in \cite{wang_lin_gu}.

\begin{rmk}
This result is really a statement about the collapse of a Tor spectral sequence, which ends up giving a K\"{u}nneth formula for KO in this particular case --- typically $KO$ lacks a K\"{u}nneth formula (see, e.g. \cite{atiyah_kunneth})
\end{rmk}

\subsection{2d charge-$2m$ Superconductors}

We compute examples from \cite{wang_charge_2m}.  We begin by considering $(2+1)$-dimensional fermionic SPT where the symmetry is a $\Z/4$ and the fermionic parity is considered as a $\Z/2$-subgroup. This is usually referred to as a $\Z/4_f$ SPT phase in the physics literature. From the considerations of \cite{freed_hopkins}, we need to examine the cobordism group $\pi_3 MG$ where $G$ is $\Spin \times_{\Z/2} \Z/4$. This group fits into an extension
\begin{equation}
\Z/2 \to G \to SO \times \Z/2. 
\end{equation}
As in the case of $\Spin^c$ \cite[10.38]{freed_hopkins} this leads to a pullback diagram
\begin{equation}
\xymatrix{
BG \ar[d]\ar[rr] &  & \ast \times BZ/2\ar[d]\\
BO \ar[rr]_{(w_1,w_2)} &  & K(\Z/2,1) \times K(\Z/2,2)
}
\end{equation}
which is in turn equivalent to a pullback diagram
\begin{equation}
\xymatrix{
  BG \ar[r] \ar[d] & B\Spin \ar[d]\\
  BSO \times B\Z/2 \ar[r]^{(\id, 2\xi)} & BSO
}
\end{equation}

where $2 \xi$ is twice the sign representation. 

The map $B\Z/2 \to BSO$ is given by $2 \xi$. The corresponding Thom space is thus
\begin{equation}
M \Spin \sma (B\Z/2)^{2\xi} \simeq \Sigma^{-2} M \Spin\sma \R P^{\infty}_2
\end{equation}
where the final identificiation is by, for example, \cite[Prop. 4.3]{atiyah_thom_complexes}. We compute the homotopy groups of this spectrum.

As usual, since we are dealing with low degrees ($\leq 8$) it will suffice to compute the homotopy groups of $ko \sma \R P^\infty_2$. To this end, we first note that the cohomology of $\R P^\infty_2$ its in the following exact sequence of $\mc{A}(1)$-modules:
\begin{equation}
H^\ast (C \eta) \to H^\ast (\R P^\infty_2) \to \Sigma^1 P 
\end{equation} 

Here $C \eta$ is the $\mc{A}(1)$-module obtained by having two elements in degree 0 and 2 attached by a $\Sq^2$. This is the cohomology $H^\ast (\C P^2)$ considered as an $\mc{A}(1)$-module, but shifted down two degrees, i.e. $\Sigma^{-2} H^\ast (\C P^2)$. This can be pleasantly visualized in Fig. \ref{fig:Ceta}

\begin{figure}
  \begin{center}
    \begin{tikzpicture}
      \fill (0, 0) circle (3pt);
      \fill (0, 2) circle (3pt);
      \sqtwoL (0, 0, black); 
    \end{tikzpicture}
      \end{center}
    \caption{\label{fig:Ceta} $H^\ast (\C P^2)$ as an $\mc{A}(1)$-module}
\end{figure}

\begin{figure}\label{fig:Ceta_LES}
  \begin{tikzpicture}[scale=.4]
    \fill (-4,0) circle (3pt);
    \fill (-4,2) circle (3pt);
    \sqtwoL(-4,0,black);
    \foreach \y in {-1,0,1, 2, 3, 4, 5, 6, 7, 8,9,10}
             {\fill (0,\y+1) circle (3pt);}
             \sqtwoL (0,1,black);
             \sq1 (0,3,black);
             \sqtwoL (0,4,black);
             \sq1 (0,5,black);
             \sqtwoR (0,5,black);
             \sq1 (0,7,black)
             \sqtwoR (0,8,black);
             \sq1 (0,9,black);
             \sqtwoL(0,9,black);
             \sq1 (0,1,black);
             \sqtwoR (0,0,black)
                 \foreach \y in {0, 2, 3, 4, 5, 6, 7, 8,9,10}
             {\fill (4,\y+1) circle (3pt);}
             \sqtwoL (4,1,black);
             \sq1 (4,3,black);
             \sqtwoL (4,4,black);
             \sq1 (4,5,black);
             \sqtwoR (4,5,black);
             \sq1 (4,7,black)
             \sqtwoR (4,8,black);
             \sq1 (4,9,black);
             \sqtwoL(4,9,black);
             \draw[red,->] (-4,2) -- (0,2);
             \draw[red,->] (-4,0) -- (0,0);
  \end{tikzpicture}
  \caption{The exact sequence $C\eta \to H^\ast (\R P^\infty_2) \to \Sigma^1 P$} 
\end{figure}

The module $P$ is the one with the periodic resolution, see Ex. \ref{periodic_resolution}. We thus obtain a long exact sequence in $\ext$:
\begin{equation}
  \xymatrix@R=.3cm{
    \ext^{s,t}(\Sigma^1 P, \mbf{F}_2) \ar[r] & \ext^{s,t} (H^\ast (\R P^\infty_2), \mbf{F}_2) \ar[r] & \ext^{s,t} (H^\ast (C \eta), \mbf{F}_2) \ar[dll]^{\delta}\\
    \ext^{s+1,t} (\Sigma^1 P, \mbf{F}_2) \ar[r] & \cdots
    }
  \end{equation}
  In order to compute $\ext^{s,t}(H^\ast (\R P^\infty_2), \mbf{F}_2)$ it therefore remains to compute $\ext^{s,t}(H^\ast(C\eta), \mbf{F}_2)$. There are two ways to do this. The first is to write out an $\mc{A}(1)$-resolution by hand, which is a pleasant enough exercise. Unfortunately, the pictures become somewhat unwieldy after a little while (though a clear pattern emerges). The other option is to note that $\mc{A}(1)\sslash E(1) \cong H^\ast(C\eta)$ where $E(1)$ is standard notation for the subalgebra of $\mc{A}$ generated by $\Sq^1$ and $Q$ with $Q_1 = \Sq^3 + \Sq^2 \Sq^1$. Thus, by the change of rings theorem \ref{change_of_rings} we have
  \begin{equation}
  \ext^{s,t}_{\mc{A}(1)}(C \eta, \mbf{F}_2) \cong \ext^{s,t}_{\mc{A}(1)} (\mc{A}(1)\sslash E(1), \mbf{F}_2) \cong \ext^{s,t}_{E(1)} (\mbf{F}_2,\mbf{F}_2) 
  \end{equation}

  and
  \begin{equation}
  \ext_{E(1)} (\mbf{F}_2, \mbf{F}_2) \cong \mbf{F}_2[h_0, v] \qquad \deg v = (1,1) 
  \end{equation}
  since $E(1)$ is an exterior algebra.

  The corresponding picture of the $E_2$-page of the Adams spectral sequence is in Figure \ref{fig:Ceta_resolution}

  \begin{figure}\label{fig:Ceta_resolution}
    \begin{tikzpicture}[scale=.5]
          \draw[step=1cm,gray,very thin] (0,0) grid (11,8);
    \foreach \y in {0, 1,2,3,4,5,6,7}
             {\fill (.5,\y+.5) circle (3pt) ; \draw (.5,\y+.5)--(0.5,\y+1.5);}
    \foreach \y in {1,2,3,4,5,6,7}
             {\fill (2.5,\y+.5) circle (3pt);\draw (2.5,\y+.5)--(2.5,\y+1.5);}
    \foreach \y in {1,2,3,4,5,6}
             {\fill (4.5,\y+1.5) circle (3pt); \draw (4.5,\y+1.5)--(4.5,\y+2.5);}
    \foreach \y in {1,2,3,4,5}
             {\fill (6.5,\y+2.5) circle (3pt); \draw (6.5,\y+2.5)--(6.5,\y+3.5);}
    \foreach \y in {1, 2, 3,4}
             {\fill (8.5,\y+3.5) circle (3pt); \draw (8.5,\y+3.5)--(8.5,\y+4.5);}
    \foreach \y in {1,2,3}
             {\fill (10.5, \y+4.5) circle (3pt); \draw (10.5,\y+4.5)--(10.5,\y+5.5);}
    \end{tikzpicture}
    \caption{The $E_2$-page of the $\mc{A}(1)$-Adams spectral sequence for $H^\ast (C \eta)$}
\end{figure}
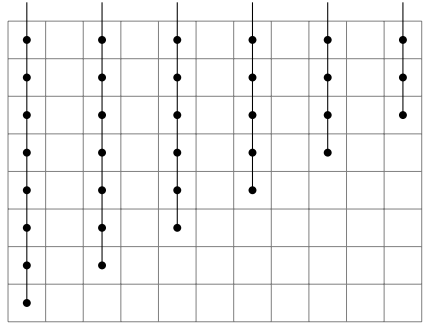

  The diagram corresponding to the Adams spectral sequence for $\pi_\ast MG$ is Fig. \ref{MG_e2_page}. We thus have
  \begin{thm}
    The low-dimensional homotopy groups of $M(\Spin \times_{\Z/2} \Z/4)$ are
    \begin{align*}
      \pi_1 M(\Spin \times_{\Z/2} \Z/4) &= \Z/4 \\
      \pi_2 M(\Spin \times_{\Z/2} \Z/4) &= 0 \\
      \pi_3 M(\Spin \times_{\Z/2} \Z/4) &= 0 \\
      \pi_4 M(\Spin \times_{\Z/2} \Z/4) &= \Z \\
      \pi_5 M(\Spin \times_{\Z/2} \Z/4) &= \Z/16
    \end{align*}
  \end{thm}

  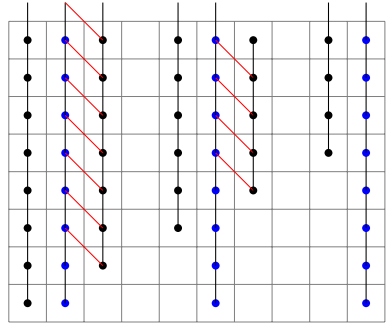
\begin{figure}
      \begin{tikzpicture}[scale=.5]
        \draw[step=1cm,gray,very thin] (0,0) grid (10,8);
        \foreach \y in {0,1,2,3,4,5,6,7}
                 {\fill (0.5,\y+.5) circle (3pt);
                   \draw (0.5,\y+.5)--(.5,\y+1.5);}
        \foreach \y in {0,1,2,3,4,5,6}
                 {\fill (2.5,\y+1.5) circle (3pt);
                   \draw (2.5,\y+1.5)--(2.5,\y+2.5);}
         \foreach \y in {0,1,2,3,4,5}
                 {\fill (4.5,\y+2.5) circle (3pt);
                   \draw (4.5,\y+2.5)--(4.5,\y+3.5);}
         \foreach \x in {1,5, 9}{
         \foreach \y in {0,1,2,3,4,5,6,7}
                  {\fill[blue] (\x+.5,\y+.5) circle (3pt);
                    \draw (\x+.5,\y+.5)--(\x+.5,\y+1.5);}}
         \foreach \y in {0,1,2,3,4,5,6}
                  {\draw[red] (2.5,\y+1.5) -- (1.5,\y+2.5);}
         \foreach \y in {0,1,2,3,4}
                  {\fill (6.5,\y+3.5) circle (3pt);}
         \foreach \y in {0,1,2,3}
                  {\draw[red] (6.5,\y+3.5)--(5.5,\y+4.5);}
         \foreach \y in {3,4,5,6}
                  {\draw (6.5,\y+.5)--(6.5,\y+1.5);}
         \foreach \y in {4,5,6,7}
                  {\fill (8.5,\y+.5) circle (3pt);
                    \draw (8.5,\y+.5) -- (8.5,\y+1.5);}
      \end{tikzpicture}
      \caption{\label{MG_e2_page}The $E_2$-page of the Adams spectral sequence for $\pi_\ast M(\Spin \times_{\Z/2} Z/4)$}
  \end{figure}

  \subsection{$\Z/2m_f$ charge superconductor}

  We now compute $M(\Spin \times_{\Z/2} \Z/2m)$ for any $m$. We first note that the computation for $m$ odd is uninteresting since the extension is trivially split in that case. We are thus reduced to computing $M(\Spin \times_{\Z/2} \Z/2^n)$. The computation proceeds in much the same way as above, but with one technical difference.

  Let $G = M (\Spin \times_{\Z/2} \Z/2^n)$. As before, this fits into an extension
  \[
  \Z/2 \to G \to SO \times \Z/2^{n-1}
  \]
  and we get a corresponding pullback diagram
  \[
  \xymatrix{
    BG \ar[d]\ar[r] & B\Spin \ar[d]\\
    BSO \times B\Z/2^{n-1} \ar[r]_{(\id, 2\xi)} & BSO
  }
  \]
  where again, $2\xi$ is twice the sign representation. This gives us the equivalence
  \[
  MG:=M(\Spin \times_{\Z/2} \Z/2^n) \simeq M\Spin \sma (B \Z/2^{n-1})^{2\xi}. 
  \]

  We do not have an analogue for Atiyah's succinct identification of the homotopy type of $(B\Z/2^{n-1})^{2\xi}$, but we need only know $H^\ast (B(\Z/2^n)^{2\xi}; \Z/2)$ as an $\mc{A}(1)$-module. For this, the Thom isomorphism suffices. The $\mc{A}(1)$-module structure of $H^\ast ((B\Z/2^n)^{2\xi};\Z/2)$ will be the same as $H^\ast (B\Z/2^n)$ but missing the bottom two cells. 

  First, we need the $\mc{A}(1)$ structure of the cohomology $H^\ast (BC_{2^n}; \Z/2)$. We note that by standard computations \cite[p.251]{hatcher_AT}, $H^\ast (BC_{2n}; \Z/2^n) \cong \Z/2^n[\alpha,\beta]/(\alpha^2 = 2\beta)$. Thus, $H^\ast (BC_{2^n}; \Z/2) \cong \Z/2[\alpha,\beta]/(\alpha^2)$. Thus, while $BC_{2^n}$ and $\R P^\infty$ have the same additive cohomology (as can easily be computed from their complexes), the ring structures differ. In particular $\Sq^1 (\alpha) = 0$. However, the cohomology ring tells us that there is a $2^n$ Bockstein: $\beta_{2^n} : H^\ast (BC_{2^n}; \Z/2^{n-1}) \to H^{\ast+1} (BC_{2^n}; \Z/2)$ maps $\alpha$ to $\beta$. See Fig. \ref{fig:bc4}. By the remarks above, $H^\ast ((B\Z/2^n)^{2\xi})$ looks like Fig. \ref{fig:bc4} but missing the bottom two dots.

  The computation of $\ext^{\ast,\ast}_{\mc{A}(1)} (H^\ast (B\Z/2^n);\mbf{F}_2)$ is now easy: as an $\mc{A}(1)$ module $H^\ast (B\Z/2^n)$ is many shifted copies of $H^\ast (C\eta)$, which we know how to resolve. The $E_2$-page of the Adams spectral sequence is pictured in the (rather crowded) Fig. \ref{fig:charge_2m}. It remains to compute the differentials. We are saved by a theorem of May and Milgram \cite{may_milgram_bockstein}. In this case it essentially tells us that the $n$th differential $d_n$ in the Adams spectral sequence is the $2^n$ Bockstein. These are the lines in red in Fig. \ref{fig:charge_2m}. 

  The diagram Fig. \ref{fig:charge_2m} gives the homotopy groups in the case of $\Z/2^3$. In general, these computations give the following theorem
  \begin{thm}
    The low dimensional homotopy groups of $M(\Spin \times_{\Z/2} \Z/2^n)$ are given by
    \begin{align*}
      \pi_0 M(\Spin \times_{\Z/2} \Z/2^n) &= \Z \\
      \pi_1 M(\Spin \times_{\Z/2} \Z/2^n) &= \Z/2^n\\
      \pi_2 M(\Spin \times_{\Z/2} \Z/2^n) &= 0\\
      \pi_3 M(\Spin \times_{\Z/2} \Z/2^n) &= \Z/2^{n-2}\\
      \pi_4 M(\Spin \times_{\Z/2} \Z/2^n) &= \Z      \\
    \end{align*}
  \end{thm}
  
  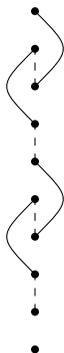
\begin{figure}
    \begin{center}
    \begin{tikzpicture}[scale=.5]
      \foreach \y in {0,1,2,3,4,5,6,7,8,9}
               {\fill (0,\y) circle (3pt);}
               \sqtwoL (0,2,black);
               \sqtwoR (0,3,black);
               \sqtwoL (0,6,black);
               \sqtwoR (0,7,black);
               \draw[dashed] (0,1) -- (0,2);
               \draw[dashed] (0,3) -- (0,4);
               \draw[dashed] (0,5) -- (0,6);
               \draw[dashed] (0,7) -- (0,8);
    \end{tikzpicture}
    \end{center}
    \caption{\label{fig:bc4} The $\mc{A}(1)$-module structure of $BZ/2^n$. The dashed lines indicate a $2^n$-Bockstein. The cohomology of the Thom space $(BC_{2^n})^{2\xi}$ will be the same, but missing the bottom two cells.}
  \end{figure}

  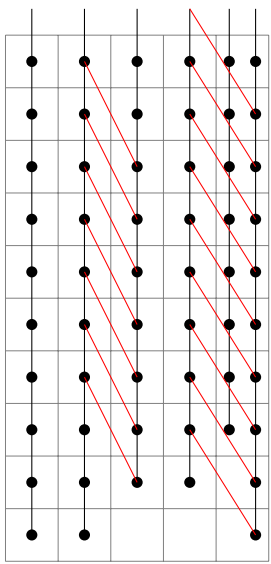
\begin{figure}
    \begin{center}
      \begin{tikzpicture}[scale=.7]
        \draw[step=1cm,gray,very thin] (0,0) grid (5,10);
        \foreach \y in {0,1,2,3,4,5,6,7,8,9}
                 {\fill (.5,\y+.5) circle (3pt);
                   \draw (.5,\y+.5) -- (.5, \y+1.5);
                   \fill (1.5,\y+.5) circle (3pt);
                   \draw (1.5,\y+.5) -- (1.5,\y+1.5);}
       \foreach \y in {1,2,3,4,5,6,7,8,9}
                {\fill (2.5, \y+.5) circle (3pt);
                  \draw (2.5,\y+.5) -- (2.5, \y+1.5);
                  \fill (3.5, \y+.5) circle (3pt);
                  \draw (3.5,\y+.5) -- (3.5, \y+1.5);}
       \foreach \y in {2,3,4,5,6,7,8,9}
                {\fill (4.25,\y+.5) circle (3pt);
                  \draw (4.25,\y+.5) -- (4.25,\y+1.5);}
        \foreach \y in {0,1,2,3,4,5,6,7,8,9}
                 {\fill (4.75,\y+.5) circle (3pt);
                   \draw (4.75,\y+.5) -- (4.75,\y+1.5);}
        \foreach \y in {1, 2, 3, 4, 5, 6, 7}
                 {\draw[red] (2.5,\y+.5) -- (1.5,\y+2.5);}
        \foreach \y in {0,1,2,3,4,5,6,7,8}
                 {\draw[red] (4.75,\y+.5) -- (3.5,\y+2.5);}
      \end{tikzpicture}
    \end{center}
    \caption{\label{fig:charge_2m} The diagram for a $\Z/2^3$ superconductor. The red lines indicate a $d_3$ coming from a $\Z/2^3$ Bockstein}
  \end{figure}

\bibliographystyle{amsplain}
\bibliography{spt}

\end{document}